\documentclass[a4paper,reqno,twoside]{amsart}
\usepackage{bbm}
\usepackage[left=3.5cm,right=3.5cm,top=3.5cm,bottom=3.5cm,footskip=1.5cm,headsep=1cm,bindingoffset=0.cm,]{geometry}

\usepackage[utf8]{inputenc}
\usepackage[english]{babel}
\usepackage[T1]{fontenc} 
\usepackage{lmodern} 
\rmfamily 
\DeclareFontShape{T1}{lmr}{b}{sc}{<->ssub*cmr/bx/sc}{}
\DeclareFontShape{T1}{lmr}{bx}{sc}{<->ssub*cmr/bx/sc}{}

\numberwithin{equation}{section}
\usepackage{amsmath}
\usepackage{amssymb}
\usepackage{amsthm}
\usepackage{amsfonts}
\usepackage{mathtools}

\usepackage{mathdots}
\usepackage{caption}
\usepackage{subcaption}

\usepackage{dsfont} 
\usepackage[format=hang]{caption}
\usepackage[normalem]{ulem}

\usepackage{standalone}
\usepackage{graphicx}
\usepackage{imakeidx}
\makeindex
\usepackage{setspace}
\usepackage{appendix}
\usepackage{pdfpages}
\usepackage{upgreek}
\usepackage{tabulary}
\usepackage{booktabs}
\usepackage{extarrows}
\usepackage{tabularx}
\usepackage{float}
\restylefloat{table}
\usepackage{enumitem}
\usepackage{csquotes}
\usepackage{bm}
\usepackage{orcidlink}
\usepackage{lipsum}

\usepackage{tikz}
\usepackage{tikz-layers}
\usepackage{tikz-cd}
\usepackage[arrow, matrix, curve]{xy}
\usetikzlibrary{matrix,positioning,decorations.pathreplacing,calc}
\usetikzlibrary{
    decorations.text,%
    decorations.markings,%
    shadows}
\usetikzlibrary{calc,intersections}
\usepackage{adjustbox}
\usepackage{graphicx}

\usepackage[
    style=numeric,
    sorting=nty,
    maxnames=99,
    maxalphanames=5,
    natbib=true,
    backend=biber,
    sortcites]{biblatex}

\DeclareNameAlias{default}{family-given}

\AtEveryBibitem{
    \clearfield{url}
    \clearfield{issn}
    \clearfield{isbn}
    \clearfield{urldate}

    \ifentrytype{book}{
        \clearfield{pages}}{%
    }
}

\usepackage{xargs}                      
\usepackage[prependcaption,textsize=tiny,textwidth=3cm]{todonotes}
\newcommandx{\unsure}[2][1=]{\todo[linecolor=red,backgroundcolor=red!25,bordercolor=red,#1]{#2}}
\newcommandx{\change}[2][1=]{\todo[linecolor=blue,backgroundcolor=blue!25,bordercolor=blue,#1]{#2}}
\newcommandx{\info}[2][1=]{\todo[linecolor=OliveGreen,backgroundcolor=OliveGreen!25,bordercolor=OliveGreen,#1]{#2}}
\newcommandx{\improvement}[2][1=]{\todo[linecolor=black,backgroundcolor=black!25,bordercolor=black,#1]{#2}}
\newcommandx{\thiswillnotshow}[2][1=]{\todo[disable,#1]{#2}}

\usepackage[noabbrev,capitalize]{cleveref}
\crefname{proposition}{Proposition}{Propositions}
\crefname{equation}{}{}

\newtheorem{theorem}{Theorem}[section]
\newtheorem{lemma}[theorem]{Lemma}
\newtheorem{proposition}[theorem]{Proposition}

\newtheorem{assumption}{Assumption}
\theoremstyle{definition}
\newtheorem{definition}[theorem]{Definition}

\newtheorem{remark}[theorem]{Remark}

\crefname{assumption}{Assumption}{Assumptions}
\crefname{definition}{Definition}{Definitions}
\crefname{corollary}{Corollary}{Corollaries}
\crefname{enumi}{item}{items}

\creflabelformat{subfigure}{#2\textsc{#1}#3}

\usepackage[final]{microtype}

\makeatletter
\newsavebox\myboxA
\newsavebox\myboxB
\newlength\mylenA

\newcommand*\xoverline[2][0.75]{%
  \sbox{\myboxA}{$\m@th#2$}%
  \setbox\myboxB\null
  \ht\myboxB=\ht\myboxA%
  \dp\myboxB=\dp\myboxA%
  \wd\myboxB=#1\wd\myboxA
  \sbox\myboxB{$\m@th\overline{\copy\myboxB}$}
  \setlength\mylenA{\the\wd\myboxA}
  \addtolength\mylenA{-\the\wd\myboxB}%
  \ifdim\wd\myboxB<\wd\myboxA%
    \rlap{\hskip 0.5\mylenA\usebox\myboxB}{\usebox\myboxA}%
  \else
    \hskip -0.5\mylenA\rlap{\usebox\myboxA}{\hskip 0.5\mylenA\usebox\myboxB}%
  \fi}
\makeatother

\usepackage{fancyhdr}

\fancypagestyle{plain}{%
    \fancyhead[C]{}
    \fancyfoot[C]{}
    
    }
\fancypagestyle{nosection}{%
    \fancyhead[CE]{}
    \fancyhead[CO]{}
    \fancyfoot[CE,CO]{\thepage}
}

\usepackage{tikz}

\usetikzlibrary{matrix} 
\usetikzlibrary{arrows,arrows.meta,bending} 

\usepackage{nicematrix}

\usetikzlibrary{hobby}

\usepackage{calc}

\renewcommand{\tilde}{\widetilde}

\renewcommand{\qedsymbol}{$\blacksquare$}

\renewcommand{\epsilon}{\varepsilon}

\renewcommand{\tilde}{\widetilde}

\usepackage{tcolorbox}
\usetikzlibrary{tikzmark,calc}

\newcommand{\on}[1]{\operatorname{#1}}

\newcommandx{\suggestions}[2][1=]{\todo[linecolor=blue,backgroundcolor=blue!25,bordercolor=blue,#1]{ #2}}

\newcommandx{\comments}[2][1=]{\todo[linecolor=red,backgroundcolor=red!25,bordercolor=red,#1]{ #2}}

\newcommandx{\silvio}[2][1=]{\todo[linecolor=blue,backgroundcolor=blue!25,bordercolor=blue,#1]{Silvio: #2}}

\newcommandx{\alex}[2][1=]{\todo[linecolor=red,backgroundcolor=red!25,bordercolor=red,#1]{Alex: #2}}

\title[Perturbative Approach to Nonlinear Capacitance Matrix]{Perturbative Approach to Nonlinear Capacitance Matrix Formulations}

\begin{document}
\author[H. Ammari]{Habib Ammari \,\orcidlink{0000-0001-7278-4877}}
\address{\parbox{\linewidth}{Habib Ammari\\
 ETH Z\"urich, Department of Mathematics, Rämistrasse 101, 8092 Z\"urich, Switzerland, \href{http://orcid.org/0000-0001-7278-4877}{orcid.org/0000-0001-7278-4877}}.}
 \email{habib.ammari@math.ethz.ch}
 \thanks{}

\author[C. Thalhammer]{Clemens Thalhammer}
\address{\parbox{\linewidth}{Clemens Thalhammer\\
 ETH Z\"urich, Department of Mathematics, Rämistrasse 101, 8092 Z\"urich, Switzerland, \href{http://orcid.org/0000-0001-7278-4877}{orcid.org/0009-0006-7218-260X}}.}
 \email{clemens.thalhammer@sam.math.ethz.ch}

 \begin{abstract}
    We study a nonlinear Helmholtz system with cubic nonlinearity on high-contrast inclusions in three dimensions, and the solitons that emerge as the contrast $\delta$ tends to zero. Using the Dirichlet-to-Neumann operator and a capacitance formalism, we develop a perturbative cascade that expands the resonant frequency and field in powers of $\sqrt{\delta}$. Our main result is a rigorous two-way correspondence with a finite discrete nonlinear capacitance system: every discrete solution lifts to a continuous soliton (a convergent expansion, analytic in $\sqrt{\delta}$), and every continuous family with the natural subwavelength scaling reduces to a discrete one. The construction is algorithmic, giving higher-order corrections in both the subwavelength and non-subwavelength regimes, the latter via a frequency-dependent capacitance matrix. We illustrate the theory numerically and characterise a symmetry-breaking bifurcation in a symmetric dimer.
 \end{abstract}

\maketitle

\par

\bigskip

\\

\tableofcontents

\section{Introduction}\label{sec:intro}
High-contrast inclusions embedded in a homogeneous medium support subwavelength
resonances: localized resonant modes whose wavelength greatly exceeds the size
of the inclusions. The canonical example is the Minnaert resonance of an air
bubble in water~\cite{minnaert1933XVI, ammariMinnaert2018}, and the phenomenon underlies much of
the recent activity in metamaterials~\cite{davies2025roadmap} and subwavelength
physics~\cite{cbms}. A central insight of this theory is
that, in the high-contrast limit $\delta \to 0$, the resonant behaviour of the
full continuous scattering problem is governed at leading order by a
finite-dimensional object, the capacitance matrix, whose eigenvalues determine
the resonant frequencies~\cite{ammari.davies.ea2024Functional}. 
 
Much less is understood once nonlinearity enters. A well-studied example\cite{angermann2023radiation,evequoz2015dual,griesmaier2022inverse} is the class of Kerr nonlinearities, where the refractive index depends on the local field
intensity~\cite{gorbach2009spatial}, and the time-harmonic problem becomes a
nonlinear Helmholtz system whose resonant modes are continuous solitons. 
A nonlinear analogue of the capacitance reduction was proposed
in~\cite{ammari2025analysis}, where existence was established only at first order
in $\delta$. This leaves two natural questions open: does the discrete model
determine the continuous soliton to all orders, with a quantitative range of
validity; and is the reduction faithful, in the sense that every continuous
soliton with the natural subwavelength scaling actually arises from the discrete
system?
 
In this paper, we answer both questions affirmatively. Recasting the transmission
problem as a weak equation on the interior domain via the exterior
Dirichlet-to-Neumann operator, following the variational approach
of~\cite{feppon.cheng.ea2023Subwavelength, fepponSubwavelength2024, angermann2023radiation}, we expand the frequency and
field as fractional power series in $\sqrt{\delta}$ and solve the resulting
cascade recursively. Our first main result,Theorem~\ref{thm:approximation_convergence}, shows
that any solution of the discrete nonlinear capacitance system lifts, under a
non-degeneracy assumption, to a continuous soliton given by a convergent
expansion with explicit recursive coefficients; in particular, it is analytic in $\sqrt{\delta}$. The principal
difficulty is controlling the convolution sums produced by the coupling of
frequency to field and by the cubic term, which we handle through a
combinatorial concentration estimate. Our second main result
Theorem~\ref{thm:converse} is the converse: every continuous family obeying
the natural bounds reduces, after passing to a subsequence, to a solution of the
discrete system. Together, these establish a two-way correspondence, showing that
the discrete model captures the nonlinear resonances faithfully.
 
We then extend the perturbative method beyond the subwavelength regime, where
the frequency is $\mathcal{O}(1)$ at leading order, the nonlinearity no longer acts as a
perturbation, and the relevant object is a frequency-dependent capacitance
matrix~\cite{ammari2026frequency}. For isolated eigenvalues, we obtain asymptotic
expansions to arbitrary order, together with an existence result for
small-amplitude solitons.
Finally, the reduction to a finite discrete system makes techniques from the
discrete self-trapping and discrete nonlinear Schr\"odinger
literature~\cite{eilbeck1985discrete,pankov2005gap, pankov2010gap} applicable to subwavelength
problems: as an illustration, we exhibit a symmetry-breaking bifurcation in a
symmetric dimer, with a criterion expressed through the eigenvalues of the
generalised capacitance matrix.

This paper is structured as follows. In Section 2, we introduce the problem formulation and some basic results about the Dirichlet-to-Neumann operator and capacitance matrices in the subwavelength regime. In Section 3, we then state our main results. Sections 4 and 5 are then devoted to the proofs of these. In Section 6, we extend the approach to the non-subwavelength regime. Section 7 serves to numerically illustrate these findings. Additionally, we show that the emergence of symmetry breaking bifurcations as an application. In Section 8, we conclude the paper with a brief summary of the findings and a further outlook.
\section{Problem Setting and Preliminaries}\label{sec:setting}

\subsection{Problem Formulation}
Let $\mathcal{I}$ be a finite or infinite index set and let $D = \{D_i\}_{i \in \mathcal{I}}$ be a collection of pairwise disjoint bounded inclusions in $\mathbb{R}^3$, each with $\mathcal{C}^1$ boundary. When $\mathcal{I}$ is infinite, we additionally require the arrangement to be periodic, that is, there exists a lattice $\Lambda \subset \mathbb{R}^3$ such that the collection $D$ is invariant under translations by elements of $\Lambda$. To each inclusion $D_i$, we assign a wave speed $v_i > 0$ and a contrast parameter $\delta_i > 0$, and denote by $v$ the wave speed in the exterior domain $\mathbb{R}^3 \setminus \overline{D}$. We set
\begin{equation}
    v_{\min} := \min_{i \in \mathcal{I}} v_i, \qquad
    v_{\max} := \max_{i \in \mathcal{I}} v_i, \qquad
    \delta := \max_{i \in \mathcal{I}} \lvert \delta_i \rvert,
\end{equation}
and assume throughout that $0 < v_{\min} \leq v_{\max} < \infty$ and $\delta < \infty$. 

We seek nontrivial pairs $(\omega, u) = (\omega(\delta), u(\delta)) \in \mathbb{C} \times H^1_{\mathrm{loc}}(\mathbb{R}^3)$ that satisfy the nonlinear Helmholtz system
\begin{equation}\label{eq:nonlinear_problem}
    \begin{cases}
        \Delta u + \dfrac{\omega^2}{v^2}\, u = 0
        & \text{in } \mathbb{R}^3 \setminus \overline{D}, \\[0.8em]
        \Delta u + \dfrac{\omega^2}{v_i^2}\, u
        + \sigma \dfrac{\omega^2}{v_i^2}\, |u|^2 u = 0
        & \text{in } D_i, \quad i \in \mathcal{I}, \\[0.8em]
        u\big|_+ = u\big|_-
        & \text{on } \partial D_i, \quad i \in \mathcal{I}, \\[0.8em]
        \delta_i \dfrac{\partial u}{\partial \nu}\bigg|_+
        = \dfrac{\partial u}{\partial \nu}\bigg|_-
        & \text{on } \partial D_i, \quad i \in \mathcal{I}, \\[0.8em]
        u \text{ is outgoing,}
    \end{cases}
\end{equation}
where $\sigma \in \mathbb{R}\setminus\{0\}$ is the nonlinearity parameter, with $\sigma < 0$ corresponding to the focusing and $\sigma > 0$ to the defocusing case, and the subscripts $\pm$ denote the limits taken from outside and inside $D_i$, respectively. We refer to such solutions as \emph{continuous solitons} with frequency $\omega^2$. The system~\eqref{eq:nonlinear_problem} arises as the time-harmonic reduction of the nonlinear wave equation~\cite{ammari2025analysis}, with cubic nonlinearity modelling a Kerr-type effect that depends on the local field intensity~\cite{gorbach2009spatial}. The high-contrast limit $\delta \to 0$ describes inclusions whose material parameters differ strongly from the surrounding medium, a regime in which subwavelength resonances are known to emerge~\cite{ammari2025analysis}. Particular attention will be paid to the \emph{subwavelength regime}, in which $\delta \to 0$ and $\omega = \mathcal{O}(\delta^{1/2})$.


\subsection{Dirichlet-to-Neumann Operator and Capacitance Formalism}
We begin by introducing the Dirichlet-to-Neumann (DtN) operator, which reduces the exterior problem to a boundary operator. Using it, we reformulate~\eqref{eq:nonlinear_problem} as a weak equation on the interior domain and introduce the capacitance matrix that governs the leading-order subwavelength behaviour.

\begin{definition}[Exterior Dirichlet-to-Neumann Operator]
    Suppose $\lvert \mathcal{I}\rvert<\infty$ and $\omega^2$ is not a {Dirichlet scattering resonance} of $D$. The exterior Dirichlet-to-Neumann operator $\mathcal{T}^\omega : H^{1/2}(\partial D) \to H^{-1/2}(\partial D)$ is defined by
    \begin{equation}
        \mathcal{T}^\omega f = \frac{\partial u}{\partial \nu}\bigg|_{\partial D},
    \end{equation}
    where $u$ is the unique solution to
    \begin{equation}
        \begin{cases}
            \Delta u + \omega^2 u = 0 & \text{in } \mathbb{R}^3 \setminus \overline{D},\\
            u\big|_+ = f & \text{on } \partial D,\\
            u \text{  is outgoing.}
        \end{cases}
    \end{equation}
\end{definition}
The following are well known properties of the DtN map~\cite{fepponSubwavelength2024, ammariFull1999,zworski} in the finite case.

\begin{proposition}\label{prop:dtn_analytic} 
    Assume that $\lvert \mathcal{I}\rvert<\infty$. Then, the Dirichlet-to-Neumann operator $\mathcal{T}^\omega : H^{1/2}(\partial D) \to H^{-1/2}(\partial D)$  is a meromorphic operator valued function in $\omega$ with poles at 
    the scattering resonances of  $-\Delta$ in $\mathbb{R}^3\setminus \overline{D}$ with Dirichlet boundary conditions on $\partial D$.
\end{proposition}
In the setting of a {crystal} of resonators {in two dimensions}, the following weaker result is known \cite[Proposition 2.2]{ammari2025analysis}(in this reference only the two-dimensional setting was considered, however using Friedrichs' inequality, see, e.g., \cite{Neas2012}, the result can easily be generalised).
\begin{proposition}\label{prop:dtn_analytic_periodic}
    Assume that $\lvert \mathcal{I}\rvert=\infty$, the resonators form a crystal and that $\omega^2_0$ is not in the spectrum of the exterior Dirichlet Laplacian. Then, the Dirichlet-to-Neumann operator $\mathcal{T}^\omega : H^{1/2}(\partial D) \to H^{-1/2}(\partial D)$  is analytic in a small neighbourhood of $\omega_0^2$.
\end{proposition}

In particular, the above results imply that the DtN operator admits a convergent power series expansion at points $\omega \in \mathbb{R}$, away from the exterior Dirichlet spectrum. It is known that for certain arrangements, the exterior Dirichlet Laplacian does exhibit spectral gaps \cite{d2023gaps, ferraresso2015singular}.

To obtain a formulation on the interior domain alone, we multiply each equation in~\eqref{eq:nonlinear_problem} by a test function $w \in \mathcal{C}^\infty_c(\mathbb{R}^3)$, integrate by parts over each $D_i$, and use the transmission conditions to replace the exterior normal derivative by $\delta\,\mathcal{T}^{\omega/v} u$ on $\partial D$. This yields the weak formulation: Find $u \in H^1(D)$ such that
\begin{equation}\label{eq:weak_formulation}
    \sum_{i\in \mathcal{I}}\int_{D_i}\nabla u \cdot\nabla \overline{w} - \frac{\omega^2}{v_i^2}\left(u\overline{w} + \sigma \lvert u\rvert^2 u \overline{w}\right)\,dx - \delta \langle \mathcal{T}^{\omega/v} u, w\rangle_{\partial D} = 0
\end{equation}
for all $w \in H^1(D)$. This formulation was first used in the one-dimensional~\cite{feppon.cheng.ea2023Subwavelength} and then extended to three dimensions in~\cite{fepponSubwavelength2024}. In these works, the key observation is that~\eqref{eq:weak_formulation} can be treated as a perturbation of the coercive sesquilinear form
\begin{equation}
    a(u,w) := \int_D \nabla u \cdot \nabla \overline{w}\,dx + \int_D u\,dx \int_D \overline{w}\,dx,
\end{equation}
whose coercivity follows from the Poincar\'e--Wirtinger inequality. In particular, the linear theory of the capacitance matrix formalism (see, e.g.,~\cite{ammari.davies.ea2024Functional,cbms}) was successfully recovered in the limit $\delta\rightarrow 0$ using this approach in~\cite{fepponSubwavelength2024, ammariNonlinear2025}.

We now introduce the capacitance matrix, which captures the leading-order coupling between the inclusions.

\begin{definition}[Capacitance Matrix]
    For each $i \in \mathcal{I}$, let $\mathbf{1}_{D_i}$ denote the characteristic function of $\partial D_i$:
    \begin{equation}
        \mathbf{1}_{D_i}(x) =
        \begin{cases}
            1, & x \in \partial D_i,\\
            0, & \text{otherwise}.
        \end{cases}
    \end{equation}
    The capacitance matrix $\mathcal{C} = (\mathcal{C}_{ij})$ is defined by
    \begin{equation}
        \mathcal{C}_{ij} = -\langle \mathcal{T}_0 \mathbf{1}_{D_i}, \mathbf{1}_{D_j}\rangle,
    \end{equation}
    where $\mathcal{T}_0 = \mathcal{T^0}$ is the Dirichlet-to-Neumann operator for $\omega=0$.
\end{definition}

There are several equivalent definitions of the capacitance matrix; it was originally derived via Gohberg--Sigal theory~\cite{ammariMinnaert2018, ammari.davies.ea2024Functional}. The formulation above, in terms of the static DtN operator $\mathcal{T}_0$, is the most natural for our purposes.
When $\lvert \mathcal{I}\rvert <\infty$, the capacitance matrix has the following properties~\cite{ammari.davies.ea2024Functional} in three dimensions.

\begin{lemma}\label{lem:cap_properties}
    The capacitance matrix $\mathcal{C}$ is symmetric and positive definite.
\end{lemma}

To state the central result of the linear theory, we introduce the mass matrix $\mathbf{M} = \mathrm{diag}\!\left(\lvert D_i\rvert / v_i^2\right)$ and the generalised capacitance matrix
\begin{equation}
    \mathcal{C}^{\mathrm{gen}} := \mathbf{M}^{-1}\mathcal{C},
\end{equation}
whose eigenvalues $\lambda_1 \leq \lambda_2 \leq \cdots$ govern the leading-order subwavelength resonances.

\begin{theorem}[Fundamental Theorem of Subwavelength Physics]\label{thm:fundamental}
    The linear subwavelength resonance problem ($\sigma = 0$)  admits $\lvert \mathcal{I}\rvert$ resonant frequencies that satisfy
    \begin{equation}
        \omega_n = \sqrt{\delta\lambda_n} + \mathcal{O}(\delta),
    \end{equation}
    where $\lambda_n$ is the $n$-th eigenvalue of $\mathcal{C}^{\mathrm{gen}}$.
\end{theorem}

If instead $\lvert \mathcal{I}\rvert = \infty$ and the arrangement is periodic, it is easy to see that the capacitance matrix takes the form of a Laurent operator. Classically, this regime has been studied by first applying the Floquet-Bloch transform to the following linear Helmholtz problem: 
\begin{equation}\label{eq:linear_problem}
    \begin{cases}
        \Delta u + \dfrac{\omega^2}{v^2}\, u = 0
        & \text{in } \mathbb{R}^3 \setminus \overline{D}, \\[0.8em]
        \Delta u + \dfrac{\omega^2}{v_i^2}\, u
         = 0
        & \text{in } D_i, \quad i \in \mathcal{I}, \\[0.8em]
        u\big|_+ = u\big|_-
        & \text{on } \partial D_i, \quad i \in \mathcal{I}, \\[0.8em]
        \delta_i \dfrac{\partial u}{\partial \nu}\bigg|_+
        = \dfrac{\partial u}{\partial \nu}\bigg|_-
        & \text{on } \partial D_i, \quad i \in \mathcal{I}, \\[0.8em]
        u \text{ satisfies an outgoing radiation condition.}
    \end{cases}
\end{equation}
Then, layer potential techniques have been used to derive a quasiperiodic capacitance matrix. Concretely, by taking the Floquet-Bloch transform of the linear problem (see \ref{eq:linear_problem}), one obtains
\begin{equation}\label{eq:linear_problemqp}
    \begin{cases}
        \Delta u + \dfrac{\omega^2}{v^2}\, u = 0
        & \text{in } Y \setminus \overline{D}, \\[0.8em]
        \Delta u + \dfrac{\omega^2}{v_i^2}\, u
         = 0
        & \text{in } D_i, \quad i \in \mathcal{I}, \\[0.8em]
        u\big|_+ = u\big|_-
        & \text{on } \partial D_i, \quad i \in \mathcal{I}, \\[0.8em]
        \delta_i \dfrac{\partial u}{\partial \nu}\bigg|_+
        = \dfrac{\partial u}{\partial \nu}\bigg|_-
        & \text{on } \partial D_i, \quad i \in \mathcal{I}, \\[0.8em]
        u \text{ is $\alpha$-quasiperiodic.}
    \end{cases}
\end{equation}
For $\alpha>0$, one can define the $\alpha$-quasiperiodic single layer potential using the Floquet-Bloch Transform of the Green function. The remaining analysis leads to the following formulation of the quasiperiodic capacitance matrix:
\begin{equation}
    \mathcal{C}^\alpha_{ij} = -\int_{\partial D_j}(\mathcal{S}_D^\alpha)^{-1}[\mathbf{1}_{D_i}]dx.
\end{equation}

Proposition \ref{prop:dtn_analytic_periodic} complements this picture by guaranteeing that the following picture commutes for resonator crystals:

\begin{figure}[!h]
\centering
\begin{tikzpicture}[
  every node/.style={align=center, font=\small},
  arr/.style={-{Stealth[length=2.5mm]}, thick},
  biarr/.style={{Stealth[length=2.5mm]}-{Stealth[length=2.5mm]}, thick}
]
  \node (UL) at (0,3) {Helmholtz system\\on $\mathbb{R}^3$};
  \node (UR) at (9,3) {Quasiperiodic\\Helmholtz system on $Y$};
  \node (LL) at (0,0) {Capacitance matrix\\on $\mathbb{Z}^d$};
  \node (LR) at (9,0) {Quasiperiodic\\capacitance matrix};

  \draw[biarr] (UL) -- node[above]{Floquet--Bloch} (UR);
  \draw[arr] (UR) -- node[sloped, above]{$\delta \to 0$} (LR);
  \draw[arr] (UL) -- node[sloped, above]{$\delta \to 0$} (LL);
  \draw[biarr] (LL) -- node[below]{Laurent symbol} (LR);
\end{tikzpicture}
\label{fig:commutative-diagram}
\end{figure}

The following result holds \cite{ammari.davies.ea2024Functional}.
\begin{theorem}
    Suppose that the fundamental cell $Y$ contains $N$ resonators, and let $\alpha>c>0$. The linear $\alpha$-quasiperiodic subwavelength resonance problem admits $N$ resonant frequencies that satisfy the asymptotic formula 
    \begin{equation}
        \omega_i^\alpha = \sqrt{\delta\lambda_i^\alpha} + \mathcal{O}(\delta^{3/2}),
    \end{equation}
    where $\lambda_i^\alpha$ is the  $i$\textsuperscript{th} eigenvalue of the $\alpha$-quasiperiodic capacitance matrix.
\end{theorem}

\subsection{A Cascade of Equations}\label{ssec:cascade}
We seek a formal solution of \eqref{eq:nonlinear_problem} in terms of fractional power series in $\delta$:
\begin{equation}\label{eq:power_series}
    u = \sum_{n=0}^\infty \delta^{n/2} u_n \quad \text{and} \quad \omega = \sum_{n=1}^\infty \delta^{n/2} \omega_n.
\end{equation}
The frequency-dependent DtN operator may be expanded as
\begin{equation}
    \mathcal{T}^{\omega/v} = \sum_{n=0}^\infty \delta^{n/2} \mathcal{A}_n, \quad \text{where} \quad \mathcal{A}_n = \sum_{k=1}^n \frac{1}{v^k}\mathcal{T}_k \sum _{i_1 + \dots+i_k=n}\prod_{j=1}^k\omega_{i_j},
\end{equation}
and $\mathcal{A}_0 = \mathcal{T}_0$. Substituting the ansatz \eqref{eq:power_series} into the weak formulation \eqref{eq:weak_formulation} and collecting terms of order $\delta^{n/2}$, we obtain the following cascade of equations:
\begin{equation}\label{eq:cascade1}
    \begin{aligned}
        \int_{D} \nabla u_n \cdot \nabla \overline{w} \, dx 
        &- \sum_{k+l+m=n}\int_D \dfrac{\omega_k\omega_l}{v_i^2} u_m \overline{w} \, dx \\
        &- \sigma \sum_{k+l+m+r+s=n}\int_D \dfrac{\omega_k\omega_l}{v_i^2}u_m\overline{u}_r u_s\overline{w} \, dx \\
        &- \sum_{k=0}^{n-2}\langle\mathcal{A}_k u_{n-k-2},w\rangle_{\partial D} = 0.
    \end{aligned}
\end{equation}

To solve \eqref{eq:cascade1} iteratively, we decompose $H^1(D) = \mathcal{V}_{d} \oplus \mathcal{V}_0$, where $\mathcal{V}_{d}$ is the space of functions that are constant on each $D_i$, and $\mathcal{V}_0$ is the space of mean-zero fluctuations.

For $n = 0$ and $n = 1$, the sums involving $\omega$ and the DtN map in \eqref{eq:cascade1} are empty. We simply obtain:
\begin{equation}
    \int_{D} \nabla u_n\cdot \nabla \overline{w} \, dx = 0 \quad \text{for all } w \in H^1(D),
\end{equation}
which implies that $u_0, u_1 \in \mathcal{V}_{d}$. Let $u_0|_{D_i} = a_i$.

Crucially, evaluating the cascade at $n=2$ and restricting the test function $w$ to the subspace $\mathcal{V}_{d}$ annihilates the gradient term ($\nabla \overline{w} = 0$). The equation evaluates precisely to
\begin{equation}
    -\int_{D} \dfrac{\omega_1^2}{v_i^2}\left(1 + \sigma\lvert u_0\rvert^2\right)u_0\overline{w} \, dx - \langle \mathcal{T}_0 u_0, w\rangle_{\partial D} = 0. 
\end{equation}
Recognizing that $\langle \mathcal{T}_0 u_0, w \rangle_{\partial D} = -w^* \mathcal{C} \mathbf{a}$ and integrating the volume terms exactly recovers the discrete nonlinear capacitance system:

\begin{equation}\label{eq:discrete_nonlinear}
        (\mathcal{C} - \omega_1^2 \mathbf{M})\mathbf{a} - \sigma \omega_1^2 \mathbf{N}(\mathbf{a}) = 0,
\end{equation}
where $\mathbf{M} = \text{diag}(|D_i|/v_i^2)$ and $\mathbf{N}(\mathbf{a})$ contains the localised cubic interactions.

For $n \geq 2$, the hierarchy is solved recursively. Assume that $u_0,\dots, u_{n-1}$ and $\omega_1,\dots,\omega_{n}$ are known. To compute $u_n$ and $\omega_{n+1}$, we proceed in three steps. Before undertaking this endeavour, we first require some more notation. We define $F_{n}\in (H^1(D))^*$ by
\begin{equation}\label{eq:remainder}
    \begin{aligned}
        \langle F_{n}, w\rangle &=\sum_{\substack{k+l+m=n\\1\leq m\leq n-3\\1\leq k,l\leq n-2}}\int_D \dfrac{\omega_k\omega_l}{v_i^2} u_m \overline{w} \, dx 
    + \sigma \sum_{\substack{k+l+m+r+s=n\\\\1\leq m,r,s\leq n-3\\1\leq k,l\leq n -2}}\int_D \dfrac{\omega_k\omega_l}{v_i^2}u_m\overline{u}_r u_s\overline{w} \, dx\\ 
     &+ \sum_{k=1}^{n-2}\langle\mathcal{A}_k u_{n-k-2},w\rangle_{\partial D}.
    \end{aligned}
\end{equation}
Note that $F_n$ depends only on $u_0,\dots,u_{n-3}$ and $\omega_1,\dots,\omega_{n-2}$. Next, we define the $\mathbb{R}$-linear operator $\mathcal{L}_{u_0}:H^1(D)\rightarrow(H^1(D))^*$ by
\begin{equation}
    \begin{aligned}
        \langle \mathcal{L}_{u_0}(\phi), w \rangle := \sum_{i \in \mathcal{I}} \int_{D_i} &\dfrac{\omega_1^2}{v_i^2}\phi\overline{w} + 2\sigma\dfrac{\omega_1^2}{v_i^2}|u_{0,d}|^2\phi\overline{w} 
        + \sigma\dfrac{\omega_1^2}{v_i^2}u_{0,d}^2\overline{\phi}\overline{w} \,dx + \langle\mathcal{T}_0\phi,w\rangle_{\partial D}.
    \end{aligned}
\end{equation}
For $u\in H^1(D)$, we denote by $u_d$ the projection of $u$ into $\mathcal{V}_d$ and by $u_0 = u-u_d\in \mathcal{V}_0$. Finally, let $B:\mathbb{C}\rightarrow(L^2(D))^*$ be defined by
\begin{equation}
    \langle B(\alpha),w\rangle=\sum_{i\in\mathcal{I}}\int_{D_i}\frac{2\alpha\omega_1}{v_i^2}(1+\sigma\lvert u_0\rvert^2)u_0\overline{w}\,dx.
\end{equation}

We thus rewrite \eqref{eq:cascade1} in the following much more succinct form:
\begin{equation}\label{eq:cascade2}
    \langle \nabla u_ n,\nabla w\rangle-\langle \mathcal{L}_{u_0}u_{n-2},w\rangle - \langle B(\alpha),w\rangle = \langle F_{n},w\rangle.
\end{equation}

\textbf{Step 1: $u_{n,0}$:} In a first step, we determine $u_{n,0} \in \mathcal{V}_0$ by testing \eqref{eq:cascade2} at level $n$ against $w \in \mathcal{V}_0$. The strict coercivity of the Laplacian on $\mathcal{V}_0$ ensures unique invertibility.

\textbf{Step 2: $\omega_{n+1}$} To compute $\omega_{n+1}$, we test equation \eqref{eq:cascade2} at level $n+2$ with $iu_0$. Because $iu_0$ lies in the kernel of $\mathcal{L}_{u_0}$, \eqref{eq:cascade2} reduces to
\begin{equation}
    \langle B(\alpha),u_0\rangle=\langle F_{n+2},u_0\rangle,
\end{equation}
for which a unique solution exists, provided that $\lVert u_0\rVert^2 + \sigma \lVert u_0\rVert_4^4\neq 0$.

\textbf{Step 3: $u_{n,d}$} In a final step, we compute the resonator-wise constant contribution $u_{n,d}$. Testing at level $n+2$ with a test function in $\mathcal{V}_d$, we are tasked with solving
\begin{equation}
    \langle \mathcal{L}_{u_0} u_{n,d},w\rangle = \langle F_{n+2} - B(\omega_{n+1}),w\rangle.
\end{equation}
Due to our choice of $\omega_{n+1}$, we may require that $u_{n,d}\perp u_0$. Thus, we need to invert the operator $\mathcal{L}_{u_0}$ on the orthogonal complement of $u_0$ in $\mathcal{V}_d$, which we denote by $\mathcal{W}_{u_0}$. Herein lies our key assumption:
\begin{assumption}\label{as:invertability}
    The operator $\mathcal{L}_{u_0}:\mathcal{W}_{u_0}\rightarrow\mathcal{W}_{u_0}$ is invertible.
\end{assumption}
\begin{remark}
    It is a classical result of the Lyapunov-Schmidt theory \cite{Kielhofer2012-hy,ambrosetti2007nonlinear} that this assumption is satisfied for $\omega_1$ in a small neighbourhood of an isolated eigenvalue of the generalised capacitance matrix. It is known that the lowest eigenvalue of the generalised capacitance matrix is always simple \cite{feppon2022modal}, so the results of Section $3$ are not vacuous.
\end{remark}

\section{Main Results}\label{sec:mainresults}

In this section, we state our main results and contrast them with existing 
work. Our central contribution is a rigorous two-way correspondence between 
the discrete nonlinear capacitance system~\eqref{eq:discrete_nonlinear} and 
the continuous nonlinear Helmholtz problem~\eqref{eq:nonlinear_problem}: 
every discrete solution lifts to a continuous one 
(assuming a non-degeneracy condition; see Theorem~\ref{thm:approximation_convergence}), and every continuous 
solution with the natural scaling reduces to a discrete one 
(Theorem~\ref{thm:converse}). Together, these establish that the discrete 
model is faithful in the sense that no solution branch is added or missed by the 
capacitance formulation reduction.

The lifting direction extends~\cite{ammari2025analysis}, where the existence of 
continuous solutions associated with the discrete system was established 
only at first order in $\delta$. We construct the continuous solution as a 
convergent fractional power series in $\sqrt{\delta}$ to arbitrary order, with 
explicit recursive coefficients and a quantitative threshold $\delta_0$; 
analyticity in $\sqrt{\delta}$ follows immediately. The reduction direction 
has no counterpart in existing work: previous results assume \emph{a 
priori} that solutions admit a convergent expansion, whereas we prove 
that every continuous family with the bounds of~\eqref{eq:uniform_bounds} 
is asymptotically discrete.

A complementary approach in the related high-refractive-index regime is 
developed in~\cite{ammari2026dielectric}, where existence is obtained via 
Lyapunov-Schmidt reduction. While this approach can also be shown to work in the present setting, the benefit of the more involved series approach is that it also yields an algorithm to compute higher order corrections.

\begin{theorem}[Discrete-to-continuous lifting]\label{thm:approximation_convergence}
    Let $(\mathbf{a}, \omega_1) \in \mathbb{R}^N \times \mathbb{R}$,
    with $\omega_1 \neq 0$, be a solution of the discrete nonlinear
    capacitance system~\eqref{eq:discrete_nonlinear}, and suppose that
    Assumption~\ref{as:invertability} holds.
    Then, there exists $\delta_0 > 0$ such that, for every
    $\delta < \delta_0$, problem~\eqref{eq:nonlinear_problem} admits
    a nontrivial solution $(u, \omega) \in H^1(D) \times \mathbb{C}$
    given by the convergent expansions
    \begin{equation}
        u = \sum_{n=0}^{\infty} \delta^{n/2} u_n,
        \qquad
        \omega = \sum_{n=1}^{\infty} \delta^{n/2} \omega_n,
    \end{equation}
    where $u_0\mid_{D_i}=a_i$. In particular, the pair $(u,\omega)$ is analytic in $\sqrt{\delta}$.
\end{theorem}

\begin{theorem}[Continuous-to-discrete reduction]\label{thm:converse}
    Let $\{(\omega^\delta, u^\delta)\}_{\delta > 0} \subset \mathbb{C} \times H^1(D)$
    be a family of nontrivial solutions to~\eqref{eq:nonlinear_problem}
    satisfying the uniform bounds
    \begin{equation}\label{eq:uniform_bounds}
        \omega^\delta \sim \delta^{1/2},
        \qquad
        \lVert u^\delta \rVert_{L^2(D)}+ \lVert u^\delta \rVert_{L^4(D)} = \mathcal{O}(1),
    \end{equation}
    as $\delta \to 0$. Write $\omega^\delta = \delta^{1/2}\omega_1^\delta + o(\delta^{1/2})$
    and let $u_0^\delta := \Pi_d \, u^\delta \in \mathcal{V}_d$ denote the
    projection onto the space of resonator-wise constant functions.

    Then, after passing to a subsequence, if necessary, there exist
    $\omega_1 \in \mathbb{C} \setminus$ and
    $\mathbf{a} = (a_1, \dots, a_N)^\top \in \mathbb{C}^N \setminus$
    such that
    \begin{equation}
        \omega_1^\delta \to \omega_1,
        \qquad
        u_0^\delta|_{D_i} \to a_i,
    \end{equation}
    as $\delta \to 0$, and the limit $(\mathbf{a}, \omega_1)$ solves the
    discrete nonlinear capacitance
    system~\eqref{eq:discrete_nonlinear}. Moreover, if $\omega_1^2\notin\sigma(\mathbf{M}^{-1}\mathcal{C})$, the limit $\mathbf{a}$ is nonzero.
\end{theorem}
\section{Proof of Theorem \ref{thm:approximation_convergence}}\label{sec:proof1}
In this section, we show that the iteratively defined sequence of Subsection \ref{ssec:cascade} converges for $\delta$ sufficiently small, which finalises the proof of Theorem \ref{thm:approximation_convergence}. The main obstacle is controlling the convolution sums that arise at each step of the iteration due to the coupling of frequency to the solutions and the nonlinearity.

Our proof strategy is inspired by \cite{sacchetti2023perturbation}; however, a key difference lies in the treatment of certain combinatorial sums. As the decay exponent
$p$ in the bound of Theorem \ref{lem:bound1} increases, the weight of such sums concentrates at the corners of the simplex $\{(i_1, \dots, i_l) : i_j \geq 1, \, \sum i_j = k\}$, where a single index carries most of the total. Exploiting this concentration yields constants that are bounded independently of
$k$, allowing us to close the induction with an explicit threshold on $\delta$.

\begin{lemma}\label{lem:bound1}
    Assume that there exist $\mu, \nu, \alpha > 0$ such that for some positive integer $p$
    the following estimates hold for $k = 0, \dots, n$:
    \begin{equation}
        \lvert \omega_k \rvert \leq \frac{\mu \exp(\alpha(k-1))}{k^p},
        \qquad
        \lVert u_k \rVert_{L^2} \leq \frac{\nu \exp(\alpha k)}{(k+1)^p},
        \qquad
        \lVert \nabla u_k \rVert_{L^2} \leq \frac{\nu \exp(\alpha k)}{(k+1)^p}.
    \end{equation}
    Then, there exists $C_p > 0$, independent of $n$, such that
    \begin{equation}
        \Bigl\lVert \sum_{\substack{k+l+m = n+2, \\ m \neq n, \\ k, l \neq n+1}}
        \frac{\omega_k \omega_l}{v_i^2} u_m \Bigr\rVert_{H^1}
        \leq \frac{C_p \, v_{\min}^{-2} \mu^2 \nu \exp(\alpha n)}{(n+1)^p}.
    \end{equation}
    Moreover, $C_p \to 0$ as $p \to \infty$.
\end{lemma}

\begin{proof}
    Applying our assumptions, we find that
    \begin{equation}
        \begin{aligned}
            \Bigl\lVert \sum_{\substack{k+l+m = n+2, \\ m \neq n, \\ k, l \neq n+1}}
            \frac{\omega_k \omega_l}{v_i^2} u_m \Bigr\rVert_{H^1}
            &\leq v_{\min}^{-2} \mu^2 \nu \exp(\alpha n)
              \sum_{\substack{k+l+m = n+2, \\ 0 \leq m \leq n-1, \\ 1 \leq k, l \leq n}}
              k^{-p} l^{-p} (m+1)^{-p} \\
            &= v_{\min}^{-2} \mu^2 \nu \exp(\alpha n)
              \sum_{\substack{k+l+m = n+3, \\ 1 \leq m \leq n, \\ 1 \leq k, l \leq n}}
              k^{-p} l^{-p} m^{-p}.
        \end{aligned}
    \end{equation}
    We are therefore interested in bounding
    \begin{equation}
        \begin{aligned}
            S_n^{(1)} &:= \sum_{\substack{k+l+m = n+3, \\ 1 \leq m \leq n, \\ 1 \leq k, l \leq n}}
              k^{-p} l^{-p} m^{-p} .\\
        \end{aligned}
    \end{equation}
    Due to symmetry of the indices, we may assume that $k\leq l \leq m$ and multiply the sum by three to obtain
    \begin{equation}
        \begin{aligned}
           S_n^{(1)} &\leq 3 \sum_{\substack{k+l+m = n+3, \\ 1 \leq m \leq l \leq k \leq n}}
              k^{-p} l^{-p} m^{-p} \\
            &\leq 3 \sum_{m=3}^{\frac{2n}{3}+2}  (n+3-m)^{-p}\sum_{k=1}^{m-1}
              k^{-p} (m-k)^{-p}.
        \end{aligned}
    \end{equation}
    We begin by bounding the inner sum. To this end, we note that $k^{-p} (m-k)^{-p}$ is maximised at $(m-1)^{-p}$. This yields
    \begin{equation}
        \begin{aligned}
            S_n^{(1)} &\leq 3 \sum_{m=3}^{\frac{2n}{3}+2} \, (m-1)^{-p+1} (n+3-m)^{-p} \\
            &\leq 3 (n+1)^{-p} \sum_{m=3}^{\frac{2n}{3}+2}
            (m-1)\left( \frac{n+1}{(m-1)(n+3-m)} \right)^{p}.
        \end{aligned}
    \end{equation}
    The fraction admits two different upper bounds. On the one hand, similar to the previous argument,
    \begin{equation}
        \frac{n+1}{(m-1)(n+3-m)} \leq \frac{n+1}{2n},
    \end{equation}
    and on the other hand, using that $m \leq \frac{2n}{3}+2$,
    \begin{equation}
        \frac{n+1}{(m-1)(n+3-m)} \leq \frac{3}{(m-1)}.
    \end{equation}
    Hence, for any $1 < q < p$,
    \begin{equation}
        \begin{aligned}
            S_n^{(1)} &\leq 3 (n+1)^{-p} \sum_{m=3}^{\frac{2n}{3}+1}(m-1)
              \left( \frac{n+1}{2n} \right)^{p-q} \left( \frac{3}{m-1} \right)^{q} \\
            &\leq 3 (n+1)^{-p} \left(\frac{n+1}{2n}\right)^{p-q} \sum_{m=2}^{\infty} m\left( \frac{3}{m} \right)^{q}.
        \end{aligned}
    \end{equation}
    For $q > 2$, the sum converges and hence
    \begin{equation}
        S_n^{(1)} \leq \left(\frac{n+1}{2n}\right)^{p-q}\frac{9 \cdot 3^q (\zeta(q-1) - 1)}{ \, (n+1)^p},
    \end{equation}
    where $\zeta(\cdot)$ denotes the Riemann zeta function. The claim follows.
\end{proof}

\begin{lemma}\label{lem:bound2}
    Under the same assumptions as in Lemma~\ref{lem:bound1}, there exists some $C_p' > 0$,
    again independent of $n$, such that
    \begin{equation}
        \Bigl\lVert \sum_{\substack{k+l+m+r+s = n+2, \\ m, r, s \neq n, \\ k, l \neq n+1}}
        \int_D \frac{\omega_k \omega_l}{v_i^2} u_m \overline{u}_r u_s \overline{w} \, dx
        \Bigr\rVert_{L^2}
        \leq \frac{C_p' \, v_{\min}^{-2} \mu^2 \nu^3 \exp(\alpha(n-2))}{(n+1)^p}.
    \end{equation}
    Moreover, $C_p'\rightarrow0$, as $p$ tends to infinity.
\end{lemma}

\begin{proof}
    Applying our assumptions together with the Sobolev-Embedding Theorem, we find that
    \begin{equation}
        \begin{aligned}
            &\Bigl\lVert \sum_{\substack{k+l+m+r+s = n+2, \\ 0 \leq m, r, s \leq n-1, \\ 1 \leq k, l \leq n}}
            \int_D \frac{\omega_k \omega_l}{v_i^2} u_m \overline{u}_r u_s \overline{w} \, dx
            \Bigr\rVert_{L^2} \\
            &\qquad \leq C_S^3 v_{\min}^{-2} \mu^2 \nu^3
              \sum_{\substack{k+l+m+r+s = n+2, \\ 0 \leq m, r, s \leq n-1, \\ 1 \leq k, l \leq n}}
              k^{-p} l^{-p} (m+1)^{-p} (r+1)^{-p} (s+1)^{-p} \\
            &\qquad = C_S^3 v_{\min}^{-2} \mu^2 \nu^3
              \sum_{\substack{k+l+m+r+s = n+5, \\ 1 \leq m, r, s \leq n, \\ 1 \leq k, l \leq n}}
              k^{-p} l^{-p} m^{-p} r^{-p} s^{-p}.
        \end{aligned}
    \end{equation}
    As before, we seek an upper bound for
    \begin{equation}
        S_n^{(2)} :=\sum_{\substack{k+l+m+r+s = n+5, \\ 1 \leq m, r, s \leq n, \\ 1 \leq k, l \leq n}}
        k^{-p} l^{-p} m^{-p} r^{-p} s^{-p}.
    \end{equation}
    We may again assume that the indices are ordered according to size, with the largest index having at least size $\frac{n}{5}+1$,
    \begin{equation}
        \begin{aligned}
            S_n^{(2)}
            &\leq 5 \sum_{k=5}^{\frac{4}{5}n+5}
              \sum_{\substack{l+m+r+s = k, \\ 1 \leq l, m, r, s \leq k}}
              (n+5-k)^{-p} l^{-p} m^{-p} r^{-p} s^{-p} \\
            &\leq 5 (n+1)^{-p}\sum_{k=5}^{\frac{4}{5}n+4}
              \binom{k-1}{3} \left(\frac{n+1}{(n+5-k)(k-3)}\right)^{p} . \
        \end{aligned}
    \end{equation}
    Proceeding as in the proof of Lemma \ref{lem:bound1}, we have two distinct upper bounds for the fraction
    \begin{equation}
    \frac{n+1}{(n+5-k)(k-3)}\leq\frac{n+1}{2n}
    \end{equation}
    and using the fact that $n+5-k\geq \frac{1}{5}n+1$ for $k\leq \frac{4}{5}n+4$, we have
    \begin{equation}
       \frac{n+1}{(n+5-k)(k-3)}\leq \frac{5}{k-3}.
    \end{equation}
    Thus,
    \begin{equation}
        \begin{aligned}
                        S_n^{(2)}&\leq 5 (n+1)^{-p} \sum_{k=5}^{\frac{4}{5}n+4}
              \binom{k-1}{3} \left( \frac{n+1}{2n} \right)^{p-q}
              \left( \frac{5}{k-3} \right)^{q} \\
            &\leq 5 (n+1)^{-p} \sum_{k=5}^{\infty}
              \binom{k-1}{3} \left( \frac{n+1}{2n} \right)^{p-q}
              \left( \frac{5}{k-3} \right)^{q}.
        \end{aligned}
    \end{equation}
    Choosing $q >4$, the sum converges, and the claim follows.
\end{proof}

Before we can proceed with bounding the term related to the DtN operator, we first require two technical Lemmas. 
\begin{lemma}\label{lem:DtN_decay}
    It holds that
    \begin{equation}
        \lVert T_l \rVert \leq C R^l.
    \end{equation}
\end{lemma}

\begin{lemma}\label{lem:sum}
    Let $k,l \in \mathbb{N}$ be such that $k\geq l$ and $p>0$. Define
    \begin{equation}
        I(p,k,l) = \sum_{\substack{i_1+\dots+i_l=k\\1\leq i_j}}\prod i_j^{-p}.
    \end{equation}
    Then, there exists $C_p$ such that
    \begin{equation}
        I(p,k,l) \leq \frac{C_p^{l-1}}{k^p}.
    \end{equation}
    In particular, one can choose
    \begin{equation}
        C_p = 2^{p+2}.
    \end{equation}
    For $l=2$, the claim also holds with
    \begin{equation}
        C_p(k) = 2\left(\frac{k}{k-1}\right)^p + 2^{p+1}(\zeta(p)-1) + \left(\frac{4}{k}\right)^p.
    \end{equation}
\end{lemma}

\begin{lemma}\label{lem:bound3}
    Under the same assumptions as in Lemmas~\ref{lem:bound1} and~\ref{lem:bound2},
    there exists some $C_{\alpha,p} > 0$, again independent of $n$, such that for $\alpha$ large enough,
    \begin{equation}
        \Bigl\lvert \sum_{k+l = n-2} \langle \mathcal{A}_k u_l, v \rangle \Bigr\rvert
        \leq \frac{C_{\alpha,p} \nu \exp(\alpha(n-2))}{(n-1)^p} \lVert v \rVert_{H^1}.
    \end{equation}
    Moreover, $C_{\alpha,p}\rightarrow 0$ as $\alpha$ tends to infinity for any fixed $p$.
\end{lemma}

\begin{proof}
    Using Lemmas \ref{lem:DtN_decay} and \ref{lem:sum}, we have
    \begin{equation}
        \begin{aligned}
            \lVert \mathcal{A}_k \rVert
            &\leq \sum_{l=1}^{k} v^{-l}\lVert T_l \rVert
              \sum_{i_1 + \dots + i_l = k} \prod_j \lvert \omega_{i_j} \rvert \\
            &\leq C_T \exp(\alpha k) \, k^{-p}
              \sum_{l=1}^{k} v^{-l}R^l \mu^l \exp(-\alpha l) \, 2^{(l-1)(p+2)} .\\
        \end{aligned}
    \end{equation}
    For $\alpha$ large enough, the above geometric sum converges, and upon making $\alpha$ even larger, tends to 0. Hence, we have
    \begin{equation}
            \lVert \mathcal{A}_k \rVert \leq C_{\alpha,p} \exp(\alpha k) \, k^{-p}.
    \end{equation}
    We therefore obtain
    \begin{equation}
        \begin{aligned}
            \Bigl\lvert \sum_{k+l = n-2} \langle \mathcal{A}_k u_l, v \rangle \Bigr\rvert
            &\leq \sum_{k+l = n-2}
              \lVert \mathcal{A}_k \rVert \, \lVert u_l \rVert_{H^1} \, \lVert v \rVert_{H^1} \\
            &\leq C_{\alpha,p} \nu \exp(\alpha(n-2)) \, \lVert v \rVert_{H^1}
              \sum_{k+l = n-2} k^{-p} (l+1)^{-p} \\
            &\leq C_{\alpha,p} \nu \exp(\alpha(n-2)) \, \lVert v \rVert_{H^1} \,
              2^{p+2} (n-1)^{-p}.
        \end{aligned}
    \end{equation}
\end{proof}

\begin{lemma}\label{lem:bound4}
    Assume that the estimates of Lemmas~\ref{lem:bound1}, \ref{lem:bound2} and~\ref{lem:bound3} are satisfied up to some $n$.
    Then it holds that
    \begin{equation}
        \lVert \nabla u_{n,0} \rVert_{L^2}
        \leq \widetilde{C}_{\alpha,p} \, \frac{\nu \exp(\alpha n)}{(n+1)^p},
    \end{equation}
    where
    \begin{equation}
        \lim_{\alpha \to \infty} \lim_{p \to \infty} \widetilde{C}_{\alpha,p} = 0.
    \end{equation}
\end{lemma}

\begin{proof}
    Recall that $u_{n,0}$ is uniquely determined by equation~\eqref{eq:cascade2}.
    In particular, because $\mathcal{V}_0 \subseteq \ker B$, we are tasked with solving
    \begin{equation}
        \langle \nabla u_{n,0}, \nabla w \rangle
        = \langle \mathcal{L}_{u_0} u_{n-2}, w \rangle + \langle F_n, w \rangle,
    \end{equation}
    for all $w \in \mathcal{V}_0$.
    Since the Laplacian is invertible on $\mathcal{V}_0$, a unique solution exists.
    Using Lemmas~\ref{lem:bound1}, \ref{lem:bound2} and~\ref{lem:bound3}, we obtain
    \begin{equation}
        \begin{aligned}
            \lVert \nabla u_{n,0} \rVert_{L^2}
            &\leq C_\Delta \bigl(
              \lVert F_n \rVert_{(H^1(D))^*}
              + C_{\mathcal{L}_{u_0}} \lVert u_{n-2} \rVert_{H^1}
            \bigr) \\
            &\leq \frac{\nu \exp(\alpha n)}{(n+1)^p} \, C_\Delta
            \Bigl(
              C_p \, v_{\min}^{-2} \mu^2
              + \sigma C_p \, v_{\min}^{-2} \mu^2 \nu^2
              + C_{\alpha,p} \\
            &\qquad\qquad
              + C_{\mathcal{L}_{u_0}}
              \Bigl( \frac{n+1}{n-1} \Bigr)^{\!p} \exp(-2\alpha)
            \Bigr),
        \end{aligned}
    \end{equation}
    where $C_p \to 0$ as $p \to \infty$ and $C_{\alpha,p} \to 0$ as $\alpha \to \infty$.
    Upon taking $p$ and $\alpha$ large enough, the claim follows.
\end{proof}

\begin{lemma}\label{lem:bound5}
    Assume that the estimates of Lemmas~\ref{lem:bound1}, \ref{lem:bound2}
    and~\ref{lem:bound3} are satisfied up to some $n$.
    Then, it holds that
    \begin{equation}
        \lvert \omega_{n+1} \rvert
        \leq \widetilde{C}_{\alpha,p} \, \frac{\mu \exp(\alpha n)}{(n+1)^p},
    \end{equation}
    where
    \begin{equation}
        \lim_{\alpha \to \infty} \lim_{p \to \infty} \widetilde{C}_{\alpha,p} = 0.
    \end{equation}
\end{lemma}

\begin{proof}
    Recall that $\omega_{n+1}$ is determined by the equation
    \begin{equation}\label{eq:omega_projection}
        \left\langle B(\alpha),
        i u_0  \right\rangle
        = \left\langle
         F_{n+2}
         i u_0 \right\rangle,
    \end{equation}
    where $F_{n+2}$ is given by equation~\eqref{eq:remainder}.
    From Lemmas~\ref{lem:bound1}, \ref{lem:bound2}, and~\ref{lem:bound3}, we infer that
    \begin{equation}
        \lvert \langle F_{n+2}, u_0 \rangle \rvert
        \leq \frac{\exp(\alpha n)}{(n+1)^p}
        \bigl(
          C_p \, v_{\min}^{-2} \mu^2 \nu
          + \sigma C_p \, v_{\min}^{-2} \mu^2 \nu^3
          + C_{\alpha,p} \nu
        \bigr)
        \lVert u_0 \rVert_{H^1},
    \end{equation}
    where $C_p \to 0$ as $p \to \infty$ and $C_{\alpha,p} \to 0$ as $\alpha \to \infty$.
    That is,
    \begin{equation}
        \lvert \omega_{n+1} \rvert
        \leq \frac{\mu \exp(\alpha n)}{(n+1)^p} \cdot
        \frac{
          \bigl(
            C_p \, v_{\min}^{-2} \mu \nu
            + \sigma C_p \, v_{\min}^{-2} \mu \nu^3
            + C_{\alpha,p} \mu^{-1} \nu
          \bigr) \lVert u_0 \rVert_{H^1}
        }{
          2 \omega_1 v_{\max}^{-2}
          \bigl( \lVert u_0 \rVert^2 + \sigma \lVert u_0 \rVert_4^4 \bigr)
        }.
    \end{equation}
    Choosing $\alpha$ and $p$ large enough yields the claim.
\end{proof}

\begin{lemma}\label{lem:bound6}
    Assume that the estimates of Lemmas~\ref{lem:bound1}, \ref{lem:bound2},
    \ref{lem:bound3}, \ref{lem:bound4} and~\ref{lem:bound5} are satisfied up to some $n$.
    Then, it holds that
    \begin{equation}
        \lVert u_{n,d} \rVert_{H^1}
        \leq \widetilde{C}_{\alpha,p} \, \frac{\nu \exp(\alpha n)}{(n+1)^p},
    \end{equation}
    where
    \begin{equation}
        \lim_{\alpha \to \infty} \lim_{p \to \infty} \widetilde{C}_{\alpha,p} = 0.
    \end{equation}
\end{lemma}

\begin{proof}
    Recall that by Assumption~\ref{as:invertability}, the operator $\mathcal{L}_{u_0}$ is invertible on $\mathcal{W}_{u_0}$.
    Hence, there exists a unique solution $u_{n,d} \in \mathcal{W}_{u_0}$ to
    \begin{equation}
        -\langle \mathcal{L}_{u_0} u_{n,d}, w \rangle
        = \langle F_{n+2}, w \rangle
          + \langle B(\omega_{n+1}), w \rangle
          + \langle \mathcal{L}_{u_0} u_{n,0}, w \rangle.
    \end{equation}
    The upper bound now follows from an argument analogous to that in
    Lemma~\ref{lem:bound5}.
\end{proof}

\section{Proof of Theorem~\ref{thm:converse}}

Let $\{(\omega^\delta, u^\delta)\}_{\delta > 0} \subset \mathbb{C} \times H^1(D)$
be a family of solutions as in the statement of Theorem~\ref{thm:converse}.
By hypothesis, there exist $C_\omega, M > 0$ such that
$|\omega^\delta| \leq C_\omega \delta^{1/2}$ and
\begin{equation}\label{eq:uniform_assumption}
    \lVert u^\delta \rVert_{L^2(D)}
    + \lVert u^\delta \rVert_{L^4(D)} \leq M.
\end{equation}

\begin{lemma}\label{lem:gradientbound}
    Under the above assumptions, for all $\delta$ sufficiently small,
    \begin{equation}
        \lVert \nabla u^\delta \rVert_{L^2(D)}^2
        \leq \frac{C_\omega^2}{v_{\min}^2}
        \bigl( \lVert u^\delta \rVert_{L^2}^2
          + \lVert u^\delta \rVert_{L^4}^4 \bigr) \delta.
    \end{equation}
\end{lemma}

\begin{proof}
    Testing the weak formulation~\eqref{eq:weak_formulation} with
    $w = u^\delta$ yields
    \begin{equation}
        \lVert \nabla u^\delta \rVert_{L^2}^2
        = \sum_{i \in \mathcal{I}} \frac{|\omega^\delta|^2}{v_i^2}
          \int_{D_i}
          \bigl( |u^\delta|^2 + \sigma |u^\delta|^4 \bigr) dx
        + \delta \operatorname{Re}
          \langle \mathcal{T}^{\omega^\delta/v} u^\delta,
          u^\delta \rangle_{\partial D}.
    \end{equation}
    Since $\mathcal{T}_0$ is negative semi-definite and
    $\mathcal{T}^{\omega^\delta/v} \to \mathcal{T}_0$ in operator norm
    as $\delta \to 0$, the quadratic form
    $\operatorname{Re}\langle \mathcal{T}^{\omega^\delta/v} \cdot,
    \cdot \rangle_{\partial D}$ remains negative semi-definite for $\delta$ sufficiently small.
    Dropping this nonpositive term and applying
    $|\omega^\delta|^2 \leq C_\omega^2 \delta$ and
    $v_i \geq v_{\min}$,
    \begin{equation}
        \lVert \nabla u^\delta \rVert_{L^2}^2
        \leq \frac{C_\omega^2 \delta}{v_{\min}^2}
        \bigl( \lVert u^\delta \rVert_{L^2}^2
          + \lVert u^\delta \rVert_{L^4}^4 \bigr). \qedhere
    \end{equation}
\end{proof}

\begin{lemma}\label{lem:approximate_discrete}
    Let $u_0^\delta := \Pi_d u^\delta \in \mathcal{V}_d$ and
    $\omega_1^\delta := \delta^{-1/2}\omega^\delta$ be as in
    Theorem~\ref{thm:converse}. Then,
    $\mathbf{a}^\delta := (u_0^\delta|_{D_i})_{i \in \mathcal{I}}$
    satisfies
    \begin{equation}
        \bigl\lVert
          (\mathcal{C} - (\omega_1^\delta)^2 \mathbf{M}) \mathbf{a}^\delta
          - \sigma (\omega_1^\delta)^2 \mathbf{N}(\mathbf{a}^\delta)
        \bigr\rVert_{\ell^2}
        \leq C
        \bigl( \lVert u^\delta \rVert_{L^2}^2
          + \lVert u^\delta \rVert_{L^4}^4 \bigr)^{1/2}
        \delta^{1/2},
    \end{equation}
    for some $C > 0$ independent of $\delta$.
\end{lemma}

\begin{proof}
    It suffices to show that
    \begin{equation}\label{eq:approx_dual}
        \bigl\lvert \bigl\langle
          (\mathcal{C} - (\omega_1^\delta)^2 \mathbf{M}) \mathbf{a}^\delta
          + \sigma (\omega_1^\delta)^2 \mathbf{N}(\mathbf{a}^\delta),
          \mathbf{w}
        \bigr\rangle \bigr\rvert
        \leq C
        \bigl( \lVert u^\delta \rVert_{L^2}^2
          + \lVert u^\delta \rVert_{L^4}^4 \bigr)^{1/2}
        \delta^{1/2}
    \end{equation}
    for all $\mathbf{w} \in \ell^2(\mathcal{I})$ with
    $\lVert \mathbf{w} \rVert_{\ell^2} = 1$. We identify $\mathbf{w}$
    with the element $w \in \mathcal{V}_d$ defined by $w|_{D_i} = w_i$,
    and write $\tilde{u}^\delta := u^\delta - u_0^\delta$. By the
    Poincar\'e--Wirtinger inequality on each $D_i$ and
    Lemma~\ref{lem:gradientbound},
    \begin{equation}\label{eq:tilde_bound}
        \lVert \tilde{u}^\delta \rVert_{L^2(D)}
        \leq C_P \lVert \nabla u^\delta \rVert_{L^2(D)}
        \leq \frac{C_P C_\omega}{v_{\min}}
        \bigl( \lVert u^\delta \rVert_{L^2}^2
          + \lVert u^\delta \rVert_{L^4}^4 \bigr)^{1/2}
        \delta^{1/2}.
    \end{equation}
    Testing the weak formulation~\eqref{eq:weak_formulation} against $w$
    and using $\nabla w = 0$ gives
    \begin{equation}\label{eq:tested_weak}
        \sum_{i \in \mathcal{I}} \frac{|\omega^\delta|^2}{v_i^2}
        \int_{D_i}
        \bigl( u^\delta + \sigma |u^\delta|^2 u^\delta \bigr)
        \overline{w}_i \, dx
        + \delta \langle \mathcal{T}^{\omega^\delta/v} u^\delta,
          w \rangle_{\partial D} = 0.
    \end{equation}
    We analyse each term, replacing $u^\delta$ by $u_0^\delta$ and
    $\mathcal{T}^{\omega^\delta/v}$ by $\mathcal{T}_0$. Throughout, $C$
    denotes a constant independent of $\delta$ that may change from
    line to line.

    For the linear volume term, since $\tilde{u}^\delta$ has zero mean
    on each $D_i$ and $w_i$ is constant on $D_i$,
    \begin{equation}
        \sum_{i \in \mathcal{I}} \frac{|\omega^\delta|^2}{v_i^2}
        \int_{D_i} u^\delta \overline{w}_i \, dx
        = \delta(\omega_1^\delta)^2
          \langle \mathbf{M} \mathbf{a}^\delta, \mathbf{w} \rangle.
    \end{equation}

    For the cubic volume term, the estimate
    $\bigl||a|^2 a - |b|^2 b\bigr| \leq 3(|a|^2 + |b|^2)|a - b|$
    and H\"older's inequality give
    \begin{equation}
        \begin{aligned}
            &\biggl\lvert
            \sum_{i \in \mathcal{I}} \frac{|\omega^\delta|^2}{v_i^2}
            \int_{D_i}
            \bigl(|u^\delta|^2 u^\delta
              - |u_0^\delta|^2 u_0^\delta\bigr)
            \overline{w}_i \, dx \biggr\rvert \\
            &\qquad\leq \frac{C_\omega^2 \delta}{v_{\min}^2}
            \bigl(
              \lVert u^\delta \rVert_{L^4}^2
              + \lVert u_0^\delta \rVert_{L^4}^2
            \bigr)
            \lVert \tilde{u}^\delta \rVert_{L^4}
            \lVert w \rVert_{L^4} \\
            &\qquad\leq C \lVert u^\delta \rVert_{L^4}^2
            \bigl( \lVert u^\delta \rVert_{L^2}^2
              + \lVert u^\delta \rVert_{L^4}^4 \bigr)^{1/2}
            \delta^{3/2},
        \end{aligned}
    \end{equation}
    where we used $\lVert u_0^\delta \rVert_{L^4}
    \leq \lVert u^\delta \rVert_{L^4}$,
    $\lVert \tilde{u}^\delta \rVert_{L^4}
    \leq \kappa_S \lVert \tilde{u}^\delta \rVert_{H^1}$,
    and~\eqref{eq:tilde_bound}. The leading-order contribution is
    $\delta \sigma (\omega_1^\delta)^2
    \langle \mathbf{N}(\mathbf{a}^\delta), \mathbf{w} \rangle$.

    For the boundary term, we write
    \begin{equation}
        \delta \langle \mathcal{T}^{\omega^\delta/v} u^\delta,
        w \rangle_{\partial D}
        = \delta \langle \mathcal{T}_0 u_0^\delta,
          w \rangle_{\partial D}
        + \delta \langle \mathcal{T}_0 \tilde{u}^\delta,
          w \rangle_{\partial D}
        + \delta \langle (\mathcal{T}^{\omega^\delta/v} - \mathcal{T}_0)
          u^\delta, w \rangle_{\partial D}.
    \end{equation}
    The first term equals
    $-\delta \langle \mathcal{C} \mathbf{a}^\delta,
    \mathbf{w} \rangle$. The second is bounded by
    \begin{equation}
        \delta \lVert \mathcal{T}_0 \rVert
        \lVert \tilde{u}^\delta \rVert_{H^1}
        \lVert w \rVert_{H^1}
        \leq C
        \bigl( \lVert u^\delta \rVert_{L^2}^2
          + \lVert u^\delta \rVert_{L^4}^4 \bigr)^{1/2}
        \delta^{3/2},
    \end{equation}
    and the third by
    \begin{equation}
        \delta \lVert \mathcal{T}^{\omega^\delta/v} - \mathcal{T}_0 \rVert
        \lVert u^\delta \rVert_{H^1}
        \lVert w \rVert_{H^1}
        \leq C
        \bigl( \lVert u^\delta \rVert_{L^2}^2
          + \lVert u^\delta \rVert_{L^4}^4 \bigr)^{1/2}
        \delta^2.
    \end{equation}

    Substituting into~\eqref{eq:tested_weak} and collecting the
    leading-order terms,
    \begin{equation}
        \delta \bigl\langle
          -\mathcal{C} \mathbf{a}^\delta
          + (\omega_1^\delta)^2 \mathbf{M} \mathbf{a}^\delta
          + \sigma (\omega_1^\delta)^2
            \mathbf{N}(\mathbf{a}^\delta),
          \mathbf{w}
        \bigr\rangle
        = O\bigl(
          (\lVert u^\delta \rVert_{L^2}^2
            + \lVert u^\delta \rVert_{L^4}^4)^{1/2}
          \delta^{3/2}
        \bigr).
    \end{equation}
    Dividing by $\delta$ yields~\eqref{eq:approx_dual}.
\end{proof}


The final step is to pass to a nontrivial limit
$(\omega_1, \mathbf{a})$. The argument is inspired
by~\cite{pankov2010gap}.

Suppose $\omega_1^2 \notin \sigma(\mathbf{M}^{-1}\mathcal{C})$.
Since $\mathbf{M}^{-1}\mathcal{C}$ is self-adjoint with respect to
the $\mathbf{M}$-weighted inner product, the spectral theorem gives
\begin{equation}
    \lVert (\mathcal{C} - (\omega_1^\delta)^2 \mathbf{M})
      \mathbf{v} \rVert_{\ell^2}
    \geq c \operatorname{dist}\bigl((\omega_1^\delta)^2,\,
      \sigma(\mathbf{M}^{-1}\mathcal{C})\bigr)
    \lVert \mathbf{v} \rVert_{\ell^2},
\end{equation}
where $c > 0$ depends only on $\mathbf{M}$. Combining this with
Lemma~\ref{lem:approximate_discrete} and the triangle inequality, we
obtain
\begin{equation}\label{eq:lowerbound1}
    c \operatorname{dist}\bigl((\omega_1^\delta)^2,\,
      \sigma(\mathbf{M}^{-1}\mathcal{C})\bigr)
    \lVert \mathbf{a}^\delta \rVert_{\ell^2}
    \leq \lVert (\mathcal{C} - (\omega_1^\delta)^2 \mathbf{M})
      \mathbf{a}^\delta \rVert_{\ell^2}
    \leq C \bigl(
      \lVert u^\delta \rVert_{L^2}^2
      + \lVert u^\delta \rVert_{L^4}^4
    \bigr)^{1/2} \delta^{1/2}
    + |\omega_1^\delta|^2
      \lVert \mathbf{N}(\mathbf{a}^\delta) \rVert_{\ell^2}.
\end{equation}
Since $\mathbf{N}(\mathbf{a})_i = |D_i| v_i^{-2} |a_i|^2 a_i$, we
have
$\lVert \mathbf{N}(\mathbf{a}^\delta) \rVert_{\ell^2}
\leq C \lVert \mathbf{a}^\delta \rVert_{\ell^\infty}^2
\lVert \mathbf{a}^\delta \rVert_{\ell^2}$.
To control the first term on the right-hand side
of~\eqref{eq:lowerbound1} in terms of the discrete unknowns, we
require the following norm equivalence.

\begin{lemma}\label{lem:normequivalence}
    Let $\mathbf{a}^\delta$ be defined as in
    Lemma~\ref{lem:approximate_discrete} and assume the hypotheses of
    Theorem~\ref{thm:converse} hold. Then there exists $C > 0$,
    independent of $\delta$, such that for $\delta$ sufficiently small,
    \begin{equation}
        \lVert u^\delta \rVert_{L^2}^2
        + \lVert u^\delta \rVert_{L^4}^4
        \leq C \bigl(
          \lVert \mathbf{a}^\delta \rVert_{\ell^2}^2
          + \lVert \mathbf{a}^\delta \rVert_{\ell^2}^8
        \bigr).
    \end{equation}
\end{lemma}

\begin{proof}
    Write $u^\delta = u_0^\delta + \tilde{u}^\delta$. Since
    $\tilde{u}^\delta$ has zero mean on each $D_i$, the
    Poincar\'e--Wirtinger inequality gives
    \begin{equation}\label{eq:normequivalence1}
        \lVert u^\delta \rVert_{L^2}^2
        = \lVert u_0^\delta \rVert_{L^2}^2
          + \lVert \tilde{u}^\delta \rVert_{L^2}^2
        \leq \lVert u_0^\delta \rVert_{L^2}^2
          + C_P^2 \lVert \nabla u^\delta \rVert_{L^2}^2.
    \end{equation}
    For the $L^4$ term, the inequality
    $(a + b)^4 \leq 8(a^4 + b^4)$ and the Sobolev and Poincar\'e
    inequalities yield
    \begin{equation}\label{eq:normequivalence2}
        \lVert u^\delta \rVert_{L^4}^4
        \leq 8 \bigl(
          \lVert u_0^\delta \rVert_{L^4}^4
          + \lVert \tilde{u}^\delta \rVert_{L^4}^4
        \bigr)
        \leq 8 \bigl(
          \lVert u_0^\delta \rVert_{L^4}^4
          + \kappa_S^4 C_P^4
            \lVert \nabla u^\delta \rVert_{L^2}^4
        \bigr).
    \end{equation}
    Substituting~\eqref{eq:normequivalence1}
    and~\eqref{eq:normequivalence2} into
    Lemma~\ref{lem:gradientbound},
    \begin{equation}
        \lVert \nabla u^\delta \rVert_{L^2}^2
        \leq C \delta \bigl(
          \lVert u_0^\delta \rVert_{L^2}^2
          + \lVert \nabla u^\delta \rVert_{L^2}^2
          + \lVert u_0^\delta \rVert_{L^4}^4
          + \lVert \nabla u^\delta \rVert_{L^2}^4
        \bigr).
    \end{equation}
    By Lemma~\ref{lem:gradientbound} and the hypotheses of
    Theorem~\ref{thm:converse},
    $\lVert \nabla u^\delta \rVert_{L^2}$ is uniformly bounded in
    $\delta$. Hence, for $\delta$ sufficiently small,
    \begin{equation}
        C \delta \bigl(
          \lVert \nabla u^\delta \rVert_{L^2}^2
          + \lVert \nabla u^\delta \rVert_{L^2}^4
        \bigr)
        \leq \tfrac{1}{2}
          \lVert \nabla u^\delta \rVert_{L^2}^2,
    \end{equation}
    and absorbing into the left-hand side gives
    \begin{equation}\label{eq:gradient_in_a}
        \lVert \nabla u^\delta \rVert_{L^2}^2
        \leq C \delta \bigl(
          \lVert u_0^\delta \rVert_{L^2}^2
          + \lVert u_0^\delta \rVert_{L^4}^4
        \bigr).
    \end{equation}
    Substituting~\eqref{eq:gradient_in_a} back
    into~\eqref{eq:normequivalence1}
    and~\eqref{eq:normequivalence2}, and using
    $\lVert u_0^\delta \rVert_{L^2} \leq C
    \lVert \mathbf{a}^\delta \rVert_{\ell^2}$
    and
    $\lVert u_0^\delta \rVert_{L^4} \leq C
    \lVert \mathbf{a}^\delta \rVert_{\ell^4}
    \leq C \lVert \mathbf{a}^\delta \rVert_{\ell^2}$,
    we obtain
    \begin{equation}
        \lVert u^\delta \rVert_{L^2}^2
        + \lVert u^\delta \rVert_{L^4}^4
        \leq C \bigl(
          \lVert \mathbf{a}^\delta \rVert_{\ell^2}^2
          + \lVert \mathbf{a}^\delta \rVert_{\ell^2}^8
        \bigr).
    \end{equation}
\end{proof}



Substituting the upper bound of Lemma~\ref{lem:normequivalence} into
Lemma~\ref{lem:approximate_discrete} and applying the spectral lower
bound for $\mathcal{C} - (\omega_1^\delta)^2\mathbf{M}$, we obtain
\begin{equation}\label{eq:key_lower}
    \operatorname{dist}\bigl((\omega_1^\delta)^2,\,
      \sigma(\mathbf{M}^{-1}\mathcal{C})\bigr)
    \lVert \mathbf{a}^\delta \rVert_{\ell^2}
    \leq C \bigl(
      \lVert \mathbf{a}^\delta \rVert_{\ell^2}^2
      + \lVert \mathbf{a}^\delta \rVert_{\ell^2}^8
    \bigr)^{1/2} \delta^{1/2}
    + C \lVert \mathbf{a}^\delta \rVert_{\ell^\infty}^2
      \lVert \mathbf{a}^\delta \rVert_{\ell^2},
\end{equation}
where $\omega_1^\delta := \delta^{-1/2}\omega^\delta$.

Suppose $\lVert \mathbf{a}^\delta \rVert_{\ell^2} \to 0$ as
$\delta \to 0$. We may divide by $\lVert \mathbf{a}^\delta \rVert_{\ell^2}$ on both sides to find that $\operatorname{dist}((\omega_1^\delta)^2,
\sigma(\mathbf{M}^{-1}\mathcal{C})) \to 0$. Thus, $\omega_1^2 := \lim (\omega_1^\delta)^2
\in \sigma(\mathbf{M}^{-1}\mathcal{C})$.

On the other hand, if 
$\omega_1^2 \notin \sigma(\mathbf{M}^{-1}\mathcal{C})$,
there exists $c > 0$ such that
$\lVert \mathbf{a}^\delta \rVert_{\ell^2} \geq c$ for all $\delta$
sufficiently small. In particular,
$\lVert \mathbf{a}^\delta \rVert_{\ell^\infty}$ is also bounded from below.

Theorem~\ref{thm:converse} now follows by distinguishing two cases.

If $|\mathcal{I}| < \infty$, then $\mathbf{a}^\delta$ is a bounded
sequence in $\mathbb{C}^{|\mathcal{I}|}$ and $\omega_1^\delta$ is a
bounded sequence in $\mathbb{C}$. Upon passing to a subsequence, both
converge: $\omega_1^\delta \to \omega_1$ and
$a_i^\delta \to a_i$ for each $i \in \mathcal{I}$.
By continuity of the discrete nonlinear map, the limit
$(\mathbf{a}, \omega_1)$ solves~\eqref{eq:discrete_nonlinear}, and it
is nontrivial since $\lVert \mathbf{a} \rVert_{\ell^2} \geq c > 0$.

If $|\mathcal{I}| = \infty$, we assume that the system is periodic. By
translation invariance, we may shift each $\mathbf{a}^\delta$ so that
$\lVert \mathbf{a}^\delta \rVert_{\ell^\infty}$ is attained within
the first unit cell. The shifted sequence is once again bounded, and we may pass to a pointwise limit. Because $\lVert \mathbf{a}^\delta \rVert_{\ell^\infty}$ is attained in the first unit cell and is lower bounded, the resulting limit must be nonzero.

\section{Beyond the Subwavelength Regime}\label{sec:nonsubwavelength}
In this section, we study the applicability of the perturbative approach beyond the subwavelength regime. As we shall soon see, the nonlinearity in the non-subwavelength regime is qualitatively different from the subwavelength regime, which is why in this case we will have to contend ourselves with a weaker result. An upshot, however, of the analysis in this section is that the perturbative approach remains applicable for isolated eigenvalues of the capacitance matrix in the subwavelength and non-subwavelength regimes, hence giving asymptotic expansions to arbitrary order.

\subsection{The Frequency Dependent Capacitance Matrix}
This section serves to briefly summarise the result of \cite{ammari2026frequencydependentcapacitancematrixformulation}. In what follows, we will restrict ourselves to the setting where $\lvert \mathcal{I}\rvert <\infty$.

We now denote by $\sigma(-\Delta_N,D_i)$ the spectrum of the Neumann Laplacian on $D_i$ and further define
\begin{equation}
    \sigma(-\Delta_N,D) :=\bigcup_{i\in\mathcal{I}}v_i^2\sigma(-\Delta_N,D_i).
\end{equation}
Since each $D_i$ is compact, we find that $\sigma_N(-\Delta_N,D)$ is a countable sequence tending to infinity. Suppose now that $\omega_0^2\in\sigma(-\Delta_N,D)$. Define 
\begin{equation}
    \mathcal{J}(\omega_0) = \{i\in\mathcal{I}: w_0^2\in v_i^2\sigma(-\Delta_N,D_i)\}.
\end{equation}
Let $m_j(\omega_0)$ denote the multiplicity of $\dfrac{\omega_0^2}{v_j^2}$ in $\sigma(-\Delta_N, D_j)$. We now choose an orthonormal basis $\{u_{j,l}\}_{l=1}^{m_j}\in H^1(D_i)$. Furthermore, note that these eigenfunctions embed naturally in $H^1(D)$ by extension by zero on the other inclusions. We are now able to define the frequency dependent capacitance matrix.
\begin{definition}[Frequency Dependent Capacitance Matrix]
    Let $m(\omega_0) = \sum_{i\in\mathcal{I}}m_i(\omega_0)$. Then, the capacitance matrix $\mathcal{C}(\omega_0)$ is the $m\times m$ matrix defined by
    \begin{equation}
        \mathcal{C}(\omega_0)_{(i,k)(j,l)}=-\int_{D_i}\mathcal{T}^{\omega_0/v}[u_{j,l}]\overline{u_{i,k}}dx.
    \end{equation}
    The generalised frequency dependent capacitance matrix is further defined by
    \begin{equation}
        \mathcal{C}^{gen}(\omega_0)=\mathbf{M}^{-1}\mathcal{C}(\omega_0),
    \end{equation}
    where $\mathbf{M}$ is the diagonal matrix with $\mathbf{M}_{(i,k)(i,k)} = \dfrac{2\omega_0}{v_i^2}$.
\end{definition}
To  give just one example that the frequency dependent capacitance matrix behaves very differently from the standard capacitance matrix, compare the following Lemma \cite[Proposition 3.16]{ammari2026frequencydependentcapacitancematrixformulation} with Lemma \ref{lem:cap_properties}.
\begin{lemma}
    $\mathcal{C}(\omega_0)$ is complex symmetric.
\end{lemma}
\begin{theorem}\label{thm:fundamental_highfreq}
    The resonant frequencies of \ref{eq:linear_problem} satisfy
    \begin{equation}
        \omega(\delta) = \omega_0 + \delta \omega_1 +\mathcal{O}(\delta^2),
    \end{equation}
    where $\omega_0^2\in\sigma(-\Delta_N,D)$ and $\omega_1$ is an eigenvalue of the generalised frequency dependent capacitance matrix.
\end{theorem}
\subsection{Asymptotics for the Linear Problem}
We now use the perturbative approach to compute the asymptotic expansion of a resonant pair $(\omega(\delta),u(\omega))$ related to a simple eigenvalue of the generalised frequency dependent capacitance matrix. Note that the analysis in this subsection applies mutatis mutandis to the subwavelength regime, for which such an analysis has already been carried out in \cite{feppon2022modal}. While in the subwavelength regime, both approaches yield essentially the same results, the perturbative one carries the advantage that it gives a systematic way of computing higher order corrections. The benefit of this becomes apparent in Section \ref{sec:numerics}, where we examine the predicted convergence rates numerically. 

We begin by recalling the linear Helmholtz problem \eqref{eq:linear_problem}.
As in Subsection \ref{ssec:cascade}, \eqref{eq:linear_problem} can also be formulated variationally
\begin{equation}\label{eq:weak_formulation_linear}
    \sum_{i\in \mathcal{I}}\int_{D_i}\nabla u \cdot\nabla \overline{w} - \frac{\omega^2}{v_i^2}u\overline{w} \,dx - \delta \langle \mathcal{T}^{\omega/v} u, w\rangle_{\partial D} = 0.
\end{equation}
Theorem \ref{thm:fundamental_highfreq} now suggests the ansatz
\begin{equation}\label{eq:power_series_highfreq}
    u = \sum_{n=0}^\infty \delta^{n} u_n \quad \text{and} \quad \omega = \sum_{n=0}^\infty \delta^{n} \omega_n.
\end{equation}
\begin{remark}
    The key difference from the subwavelength regime is that the expansion is now in $\delta$ as opposed to $\delta^{1/2}$. Moreover, $\omega_0$ is no longer assumed to be zero. 
\end{remark}
Plugging this ansatz into equation \eqref{eq:weak_formulation_linear} and expanding around some $\omega_0\in\mathbb{C}$ yields
\begin{equation}\label{eq:cascade_linear1}
    \begin{aligned}
        \int_{D} \nabla u_n \cdot \nabla \overline{w} \, dx 
        &- \sum_{k+l+m=n}\int_D \dfrac{\omega_k\omega_l}{v_i^2} u_m \overline{w} \, dx - \sum_{k=0}^{n-1}\langle\mathcal{A}_k^{\omega_0} u_{n-k-1},w\rangle_{\partial D} = 0,
    \end{aligned}
\end{equation}
where the operators $\mathcal{A}_n^{\omega_0}$ result from expanding the DtN operator:
\begin{equation}
    \mathcal{T}^{\omega/v} = \sum_{n=0}^\infty \delta^{n} \mathcal{A}_n^{\omega_0}, \quad \text{where} \quad \mathcal{A}_n^{\omega_0} = \sum_{k=1}^n \frac{1}{v^k}\mathcal{T}_k^{\omega_0} \sum _{i_1 + \dots+i_k=n}\prod_{j=1}^k\omega_{i_j}.
\end{equation}
The equations at the first two orders are
\begin{align}
    &\int_{D} \nabla u_0 \cdot \nabla \overline{w} -\dfrac{\omega_0^2}{v_i^2} u_0 \overline{w} \, dx  = 0,\\
    &\int_{D} \nabla u_1 \cdot \nabla \overline{w} -\dfrac{\omega_0^2}{v_i^2} u_1\overline{w} - \dfrac{2\omega_0\omega_1}{v_i^2}u_0 \overline{w} \, dx  -\langle \mathcal{T}^{\omega_0/v}u_0,v\rangle= 0.
\end{align}
Consequently, $\omega_0$ must be associated with a Neumann eigenvalue and $u_0$ must be a Neumann eigenfunction of the Laplacian in $D$. Testing the equation at first order with another Neumann eigenfunction yields
\begin{equation}
    -\int_{D} \dfrac{2\omega_0\omega_1}{v_i^2}u_0 \overline{w} \, dx  -\langle \mathcal{T}^{\omega_0/v}u_0,w\rangle= 0.
\end{equation}
We thus find that $(\omega_1,u_0)$ must be an eigenvalue-eigenvector pair for the generalised frequency dependent capacitance matrix. Now, let
\begin{equation}\label{eq:cascade_linear1.5}
    \langle F_{n},w\rangle = 
    - \sum_{\substack{k + l + m = n \\ k,l< n\\m< n-1}}\int_D \dfrac{\omega_k\omega_l}{v_i^2} u_m \overline{w} \, dx - \sum_{k=1}^{n-1}\langle\mathcal{A}_k u_{n-k-1},w\rangle_{\partial D}.
\end{equation}
Observe that $F_{n}$ depends only on $u_0,\dots,u_{n-2}$ and $\omega_0,\dots,\omega_{n-1}$. This allows us to rewrite \eqref{eq:cascade_linear1}
\begin{equation}\label{eq:cascade_linear2}
    \int_{D} \nabla u_n \cdot \nabla \overline{w} \, dx 
    - \dfrac{\omega_0^2}{v_i^2}u_n\overline{w}- \dfrac{2\omega_0\omega_1}{v_i^2}u_{n-1}\overline{w} - \dfrac{2\omega_0\omega_n}{v_i^2}u_0\overline{w} \, dx - \langle\mathcal{T}^{\omega_0/v}u_{n-1},w\rangle_{\partial D} =  \langle F_{n},w\rangle.
\end{equation}
Next, we proceed as in Subsection \ref{ssec:cascade} and decompose $H^1(D) = \mathcal{V}_d\oplus\mathcal{V}_0$, where we now let $\mathcal{V}_d$ be the space spanned by the Neumann eigenfunctions associated with $\omega_0$ and $\mathcal{V}_0$ its orthogonal complement. Assume that $u_0,\dots,u_{n-1}$ and $\omega_0,\dots,\omega_{n}$ are known and that the following holds.
\begin{assumption}\label{as:simplicity}
    $\omega_1$ is a simple eigenvalue of $\mathcal{C}(\omega_0)^{gen}$.
\end{assumption}
We then proceed as follows:

\textbf{Step 1: $u_{n,0}$:} In a first step, we determine $u_{n,0} \in \mathcal{V}_0$ by testing \eqref{eq:cascade_linear2} against $w \in \mathcal{V}_0$ at level $n$. The strict coercivity of the Laplacian on $\mathcal{V}_0$ ensures unique invertibility.

\textbf{Step 2: $\omega_{n+1}$} To compute $\omega_{n+1}$, we test equation \eqref{eq:cascade_linear2} with $u_0$ at level $n+1$. Because $\omega_1$ is simple, we have 
\begin{equation}
    \langle \mathcal{T}^{\omega_0/v}w,u_0\rangle_{\partial D}=0,    
\end{equation}
for all $w \in \mathcal{V}_d$. Thus, we find that
\begin{equation}
    \int_{D}  - \dfrac{2\omega_0\omega_{n+1}}{v_i^2}\lvert u_0\rvert^2 \, dx =  \langle \mathcal{T}^{\omega_0/v}u_{n,0},u_0\rangle + \langle F_{n},u_0\rangle,
\end{equation}
which uniquely determines $\omega_{n+1}$.

\textbf{Step 3: $u_{n,d}$} Once again we test at level $n+1$ with $w\in \mathcal{V}_0$ to obtain that
\begin{equation}
    \int_{D} - \dfrac{2\omega_0\omega_1}{v_i^2}u_{n}\overline{w}  \, dx - \langle\mathcal{T}^{\omega_0/v}u^0_{n},w\rangle_{\partial D} =  \langle F_{n+1},w\rangle +\frac{2\omega_0\omega_n}{v_i^2}u_0\overline{w} + \langle\mathcal{T}^{\omega_0/v}u^d_{n},w\rangle_{\partial D}.
\end{equation}
Because the right hand side is orthogonal to $u_0$, complex symmetry of the frequency dependent capacitance matrix ensures that a unique solution $u^0_{n}$ exists in $\mathcal{V}_0\cap (u_0)^\perp$. 

It remains to prove that this power series expansion converges for $\delta$ small enough. Because this time $\omega_0\neq 0$, a bit more care is needed in bounding $\lVert F_{n}\rVert_{H^{-1}}$. We choose the ansatz
\begin{equation}
        \lvert \omega_k \rvert \leq \frac{\mu \exp(\alpha(k-1))}{k^p},
        \qquad
        \lVert u_k \rVert_{L^2} \leq \frac{\nu \exp(\alpha k)}{(k+1)^p},
        \qquad
        \lVert \nabla u_k \rVert_{L^2} \leq \frac{\nu \exp(\alpha k)}{(k+1)^p}.
\end{equation}
Unlike the main section, the estimate on the second term of equation \eqref{eq:cascade_linear1.5} follows readily from a slight modification of Lemma \ref{lem:bound3}. In bounding the first term, Lemma \ref{lem:bound1} cannot be applied immediately due to the presence of terms including $\omega_0$. However, we can handle these separately to obtain the following estimate:
\begin{equation}
    \sum_{\substack{l+m = n\\
    1\leq l \leq n-1\\1\leq m \leq n-2}}\int_D\frac{\omega_0\omega_l}{v_i^2}u_m\overline{w}\,dx\leq\\mu^2\nu \frac{C_p\exp(\alpha(n-1))}{(n+1)^p}\lVert w\rVert_{L^2},
\end{equation}
where $C_p$ is obtained upon a minor modification of Lemma \ref{lem:sum}. In particular, $C_p$ again tends to zero as p increases. Thus, the following result holds. 

\begin{theorem}\label{thm:approximation_convergence_linear}
    Let $\omega_0\in\sigma(-\Delta_N,D)$ and $(\mathbf{a}, \omega_1) \in \mathbb{C}^N \times \mathbb{C}$ be an eigenvector-eigenvalue pair for the generalised frequency dependent capacitance matrix at $\omega_0$ and suppose that Assumption \ref{as:simplicity} holds.
    Then, there exists $\delta_0 > 0$ such that, for every
    $\delta < \delta_0$, problem~\eqref{eq:linear_problem} admits
    a nontrivial solution $(u, \omega) \in H^1(D) \times \mathbb{C}$
    given by the convergent expansions
    \begin{equation}
        u = \sum_{n=0}^{\infty} \delta^{n} u_n,
        \qquad
        \omega = \sum_{n=0}^{\infty} \delta^{n} \omega_n,
    \end{equation}
    where $\mathbf{a}$ determines the projection of $u$ into the space of Neumann eigenfunctions. In particular, the pair $(u,\omega)$ is analytic in $\delta$.
\end{theorem}

\subsection{Existence of Small Amplitude Solutions for the Nonlinear Problem}
In this subsection we again study \eqref{eq:nonlinear_problem}, however, in the non-subwavelength regime.  This problem is qualitatively very different from the subwavelength regime: Because $\omega$ is $\mathcal{O}(1)$ at the leading order, the nonlinearity no longer acts as a perturbation of the linear equation. That is, substituting the ansatz \eqref{eq:power_series_highfreq} into the nonlinear problem, we obtain at the leading order that
\begin{equation}
    \int_{D}\nabla u_0 \cdot \nabla \overline{w} - \dfrac{\omega_0^2}{v_i^2}(1 + \sigma \lvert u_0\rvert^2)u_0\overline{w}=0.
\end{equation}
From this, it becomes apparent that there is no reason why the leading order term should be governed by the discrete problem. We therefore pursue the more modest goal of finding solitons whose magnitude is of order $\delta$, that is, taking $u_0=0$. This ensures that the nonlinear term is of the same order as the DtN operator. Substituting this ansatz into equation \eqref{eq:weak_formulation} yields
\begin{equation}\label{eq:cascade1NS}
    \begin{aligned}
        \int_{D} \nabla u_n \cdot \nabla \overline{w} \, dx 
        &- \sum_{k+l+m=n}\int_D \dfrac{\omega_k\omega_l}{v_i^2} u_m \overline{w} \, dx \\
        &- \sigma \sum_{k+l+m+r+s=n}\int_D \dfrac{\omega_k\omega_l}{v_i^2}u_m\overline{u}_r u_s\overline{w} \, dx \\
        &- \sum_{k=0}^{n-2}\langle\mathcal{A}_k u_{n-k-1},w\rangle_{\partial D} = 0.
    \end{aligned}
\end{equation}
We now consider the first few levels of the cascade
\begin{equation}
    \begin{aligned}
    &\int_{D} \nabla u_1 \cdot \nabla \overline{w} -\dfrac{\omega_0^2}{v_i^2} u_1\overline{w}  \, dx  = 0\\
    &\int_{D} \nabla u_2 \cdot \nabla \overline{w} -\dfrac{\omega_0^2}{v_i^2} u_2\overline{w}  -\frac{2\omega_0\omega_1}{v_i^2}u_1\overline{w} \, dx -\langle \mathcal{T}^{\omega_0/v}u_1,v\rangle= 0\\
    &\int_{D} \nabla u_3 \cdot \nabla \overline{w} -\dfrac{\omega_0^2}{v_i^2} u_3\overline{w}-\frac{2\omega_0\omega_1}{v_i^2}u_2\overline{w} -\frac{\omega_1^2}{v_i^2}u_1\overline{w}-\frac{2\omega_0\omega_2}{v_i^2}u_1\overline{w}- \sigma\frac{\omega_0^2}{v_i^2}\lvert u_1\rvert^2u_1\overline{w} \, dx\\
    &-\langle \mathcal{T}^{\omega_0/v}u_2,v\rangle-\langle \mathcal{A}_1^{\omega_0/v}u_1,v\rangle= 0.\\
\end{aligned}
\end{equation}
This ansatz has the curious feature that the nonlinearity only enters at level $3$, yet the cascade remains solvable. Once again, we define an error term
\begin{equation}\label{eq:cascade1.5NS}
    \langle F_{n},w\rangle = 
    - \sum_{\substack{k + l + m = n \\ k,l< n-1\\m< n-1}}\int_D \dfrac{\omega_k\omega_l}{v_i^2} u_m \overline{w} \, dx- \sigma \sum_{\substack{k+l+m+r+s=n\\m,r,s\geq 1}}\int_D \dfrac{\omega_k\omega_l}{v_i^2}u_m\overline{u}_r u_s\overline{w} \, dx - \sum_{k=1}^{n-2}\langle\mathcal{A}_k u_{n-k-1},w\rangle_{\partial D}.
\end{equation}
Note that $F_{n-1}$ depends only on $u_1, \dots, u_{n-2}$ and $\omega_0,\dots,\omega_{n-2}$. Using this, we can rewrite equation \eqref{eq:cascade1NS} as
\begin{equation}\label{eq:cascade2NS}
    \int_{D} \nabla u_n \cdot \nabla \overline{w}
    - \dfrac{\omega_0^2}{v_i^2}u_n\overline{w}- \dfrac{2\omega_0\omega_1}{v_i^2}u_{n-1}\overline{w} - \frac{2\omega_0\omega_{n-1}}{v_i^2}u_1\overline{w}\, dx - \langle\mathcal{T}^{\omega_0/v}u_{n-1},w\rangle_{\partial D} =  \langle F_{n-1},w\rangle.
\end{equation}
We first examine the solvability of the system. Assume that $u_1,\dots,u_{n-1};\omega_1,\dots,\omega_{n-2}$ are known and that Assumptions \ref{as:simplicity} hold.

\textbf{Step 1: $u_{n,0}$:} In a first step, we determine $u_{n,0} \in \mathcal{V}_0$ by testing \eqref{eq:cascade2NS} at level $n$ against $w \in \mathcal{V}_0$. The strict coercivity of the Laplacian in $\mathcal{V}_0$ ensures unique invertibility.

\textbf{Step 2: $\omega_{n-1}$} To compute $\omega_{n-1}$, we test equation \eqref{eq:cascade2NS} at level $n$ with $u_0$, yielding
\begin{equation}
    -\int_{D} \frac{2\omega_0\omega_{n-1}}{v_i^2}\lvert u_{1}\rvert^2\, dx =\langle \mathcal{T}^{\omega/v}u_1,u_1\rangle_{\partial D} + \langle F_{n-1},u_1\rangle, 
\end{equation}
for which a unique solution exists.

\textbf{Step 3: $u_{n,d}$} In a final step, we test equation \eqref{eq:cascade2NS} with $w\in\mathcal{V}_d$ at level $n+1$ to compute $u_{n,d}$. This yields
\begin{equation}\label{eq:cascade3NS}
    \int_{D} 
    - \dfrac{2\omega_0\omega_1}{v_i^2}u_{n}\overline{w} - \frac{2\omega_0\omega_{n}}{v_i^2}u_1\overline{w}\, dx - \langle\mathcal{T}^{\omega_0/v}u_{n},w\rangle_{\partial D} =  \langle F_{n},w\rangle.
\end{equation}
By construction, the above holds for $u_1$ independent of $u_{n}$ since $\omega_1$ is simple, it follows that the above has a unique solution $u_{n,d}$ that is perpendicular to $u_1$.

Next, we have to consider the growth of the coefficients. To this end, one chooses the ansatz
\begin{equation}
    \lvert \omega_k \rvert \leq \frac{\mu \exp(\alpha(k-1))}{k^p},
        \qquad
        \lVert u_k \rVert_{L^2} \leq \frac{\nu \exp(\alpha (k-1))}{k^p},
        \qquad
        \lVert \nabla u_k \rVert_{L^2} \leq \frac{\nu \exp(\alpha (k-1))}{k^p}.
\end{equation}
The strategy of the preceding sections revolving around avoiding the corner of the simplexes in the sums no longer works. However, the index shift in the above ansatz grants a factor of $\exp(-\alpha)$ compared to Lemmas \ref{lem:bound1} and \ref{lem:bound2}, which can be used to control the error terms; for reference, see the technique employed in \cite{sacchetti2023perturbation}. Thus, we obtain our final theorem.
\begin{theorem}\label{thm:approximation_convergence_nonlinear}
    Let $\omega_0\in\sigma(-\Delta_N,D)$ and $(\mathbf{a}, \omega_1) \in \mathbb{C}^N \times \mathbb{C}$ be an eigenvector-eigenvalue pair for the generalised frequency dependent capacitance matrix at $\omega_0$ and suppose that Assumption \ref{as:simplicity} holds.
    Then, there exists $\delta_0 > 0$ such that, for every
    $\delta < \delta_0$, problem~\eqref{eq:nonlinear_problem} admits
    a nontrivial solution $(u, \omega) \in H^1(D) \times \mathbb{C}$
    given by the convergent expansions
    \begin{equation}
        u = \sum_{n=1}^{\infty} \delta^{n} u_n,
        \qquad
        \omega = \sum_{n=0}^{\infty} \delta^{n} \omega_n,
    \end{equation}
    where $\mathbf{a}$ determines the projection of $u$ into the space of Neumann eigenfunctions. In particular, the pair $(u,\omega)$ is analytic in $\delta$.
\end{theorem}

\section{Numerical Illustrations}\label{sec:numerics} 
This section presents some numerical illustrations, underlining the theoretical findings. We begin with a description of the numerical implementation of Section \ref{ssec:cascade}.

Each cascade order requires a Neumann--Poisson solver on every resonator
together with the exterior coupling. The Poisson problems are discretised
spectrally: spherical harmonics $Y_\ell^m$ up to degree $\ell_{\max}$ for the
angular variables, and Chebyshev--Gauss--Lobatto collocation with $k_{\max}$
nodes in the radius. As the angular Laplacian is diagonal in $Y_\ell^m$, each
mode reduces to a one-dimensional radial operator that is LU-factorised once
(with the Neumann condition and the regularity/zero-mean constraint imposed on
the boundary rows) and reused across orders; the Chebyshev basis converges
exponentially in $k_{\max}$, so the spatial error lies below the
cascade-truncation error even for modest $k_{\max}$.

The exterior is eliminated through the Dirichlet-to-Neumann operator
$T(\omega)$, built from the outgoing multipole expansion and the Helmholtz
addition theorem; its static
limit gives the capacitance matrix, and $T(\omega)$ is Taylor-expanded in
$\omega$ so each order applies precomputed coefficients. The leading-order
nonlinear capacitance eigenproblem is solved by Newton's method, after which
each higher order reduces to one back-substitution per resonator and a small
dense solve. All results use $\ell_{\max}=8$ and $k_{\max}=16$, for which both
truncations converge below the reported errors; the accuracy is verified
against exact references when available
and otherwise through the $L^2$ residuals of the interior equation and the
transmission conditions.

We now begin with a simple sanity check: In Figure \ref{fig:single_resonator_linear} we consider a single sphere. Here, an exact formula for the subwavelength resonance is known \cite{ammariMinnaert2018}.  In Subfigures \ref{fig:single_resonator_linear1} and \ref{fig:single_resonator_linear2}, we see excellent agreement with the predicted convergence rates of Theorem \ref{thm:approximation_convergence_linear}. 

\begin{figure}[htbp]
     \centering
     \begin{subfigure}[b]{0.45\textwidth}
         \centering
         \includegraphics[width=\textwidth]{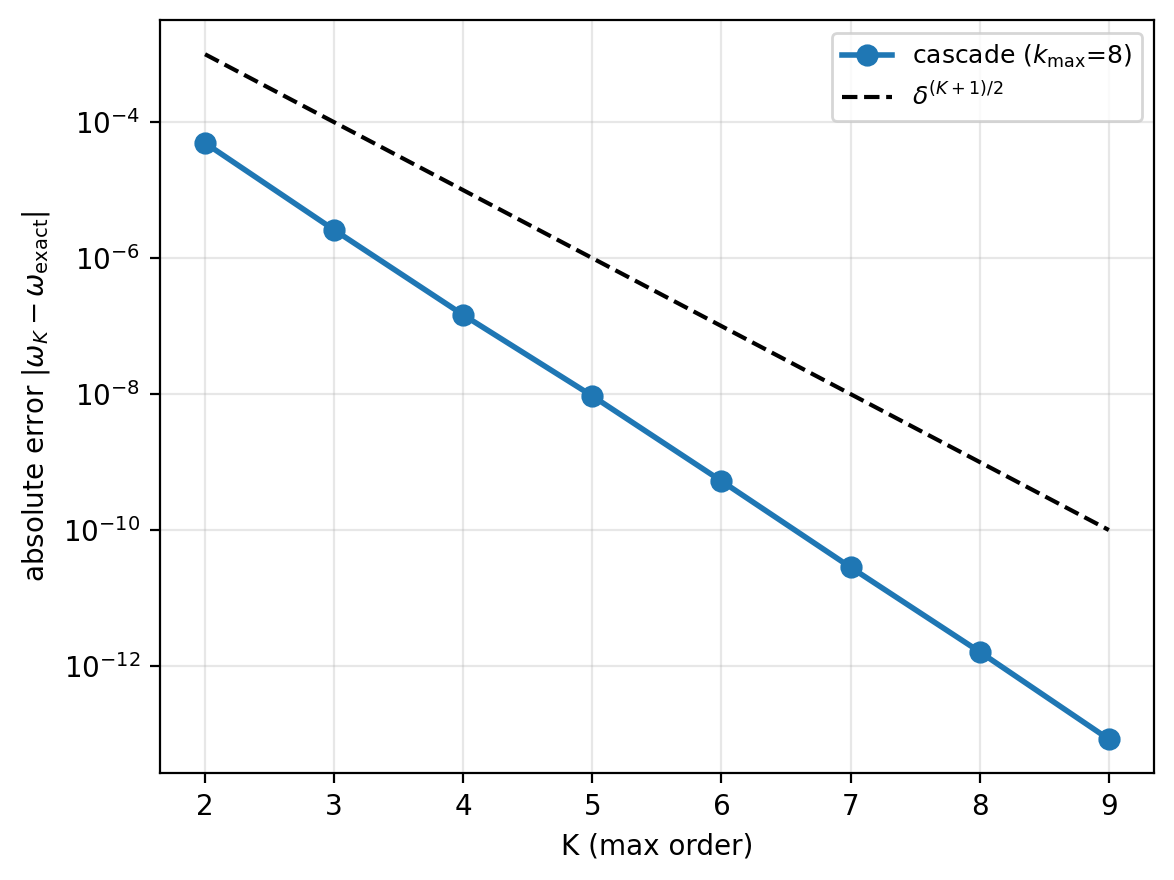}
         \caption{K-convergence of $\omega$ at $\delta = 0.01$}
         \label{fig:single_resonator_linear1}
     \end{subfigure}
     \hfill
     \begin{subfigure}[b]{0.45\textwidth}
         \centering
         \includegraphics[width=\textwidth]{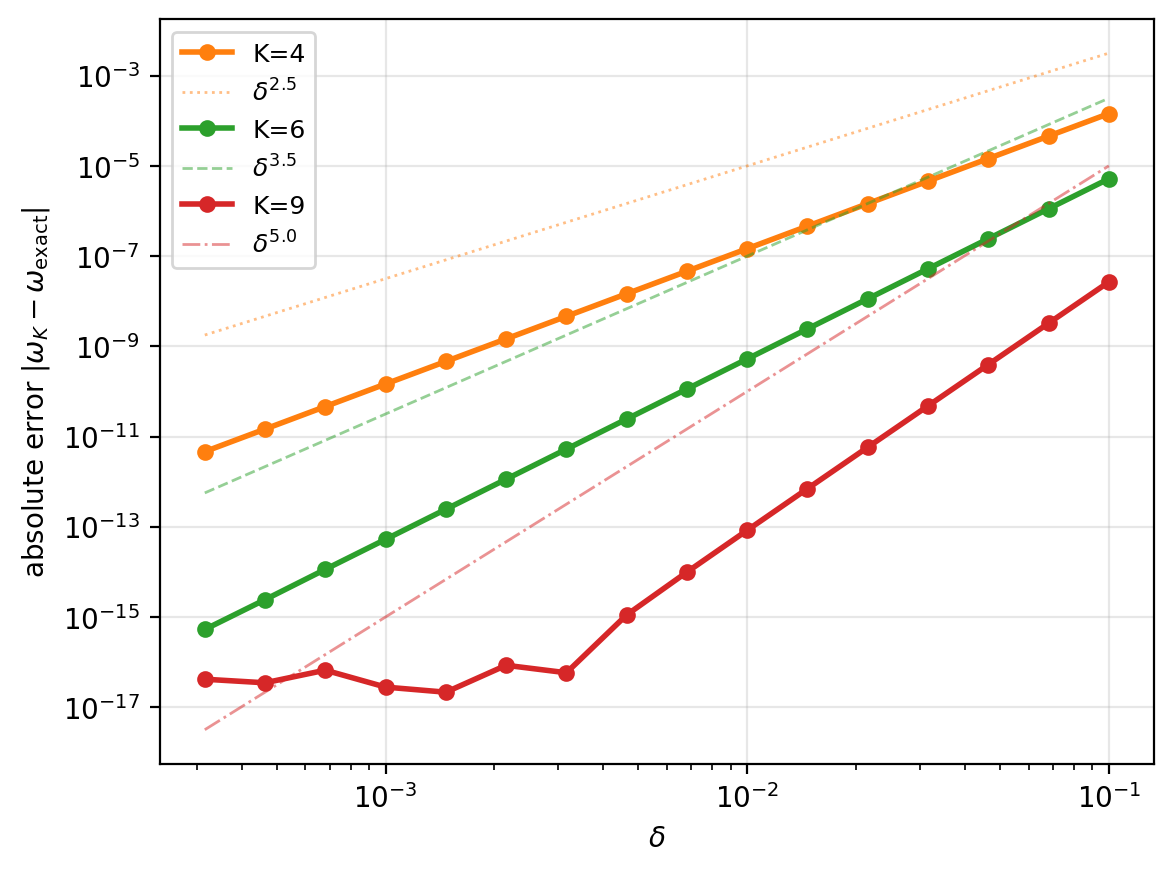}
         \caption{$\delta$-convergence}
         \label{fig:single_resonator_linear2}
     \end{subfigure}

     \vspace{1em} 

     \begin{subfigure}[b]{0.45\textwidth}
         \centering
         \includegraphics[width=\textwidth]{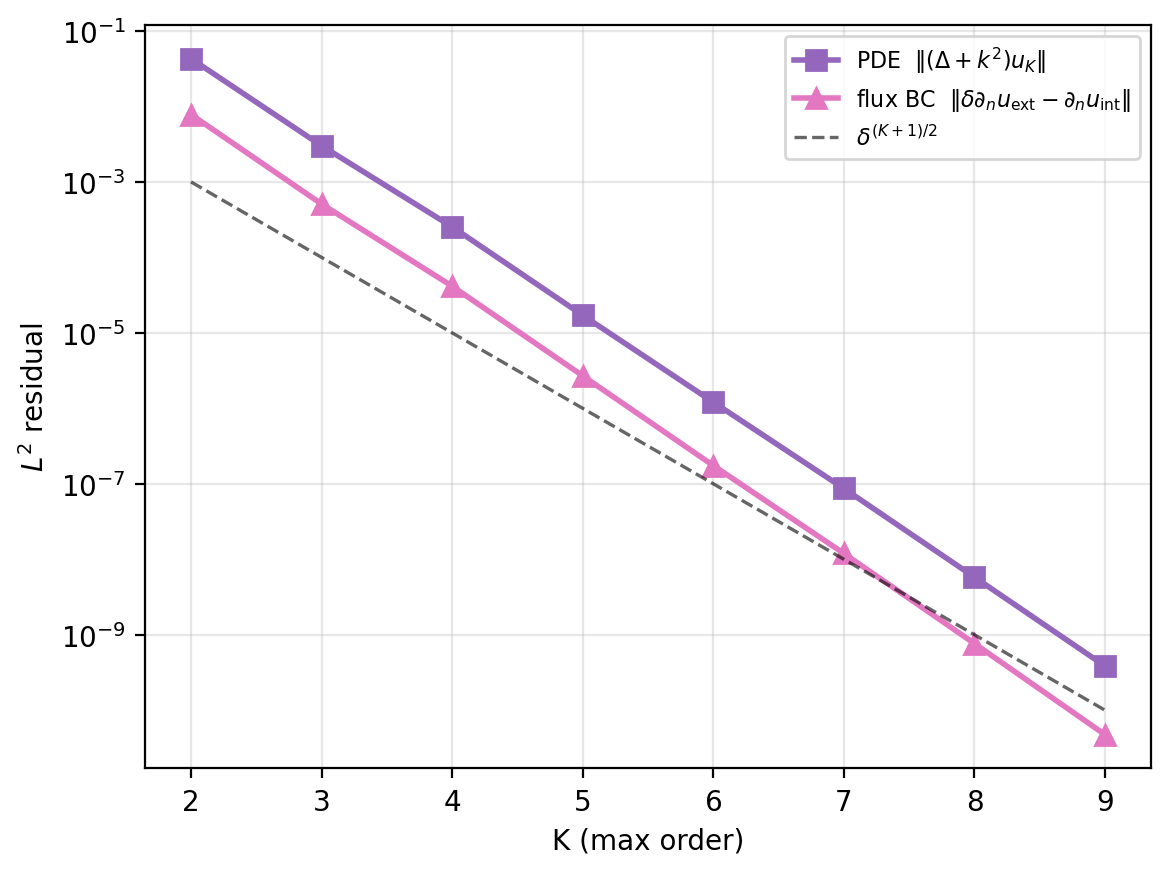}
         \caption{Operator residuals at $\delta = 0.01$}
         \label{fig:single_resonator_linear3}
     \end{subfigure}
     
     \caption{Comparison of perturbative approach versus exact solution for a single resonator with radius $r=0.1$, interior wave speed $v_1 = 1/3.48$ and exterior wave speed $v=1$.}
     \label{fig:single_resonator_linear}
\end{figure}

Next, we consider a two resonator system. Here, no exact formulas for the two subwavelength resonances are available. Therefore, we compute a reference solution using a truncated multipole method. The resulting plots can be seen in Figure \ref{fig:dimer_linear}. 

\begin{figure}[htbp]
     \centering
     \begin{subfigure}[b]{0.45\textwidth}
         \centering
         \includegraphics[width=\textwidth]{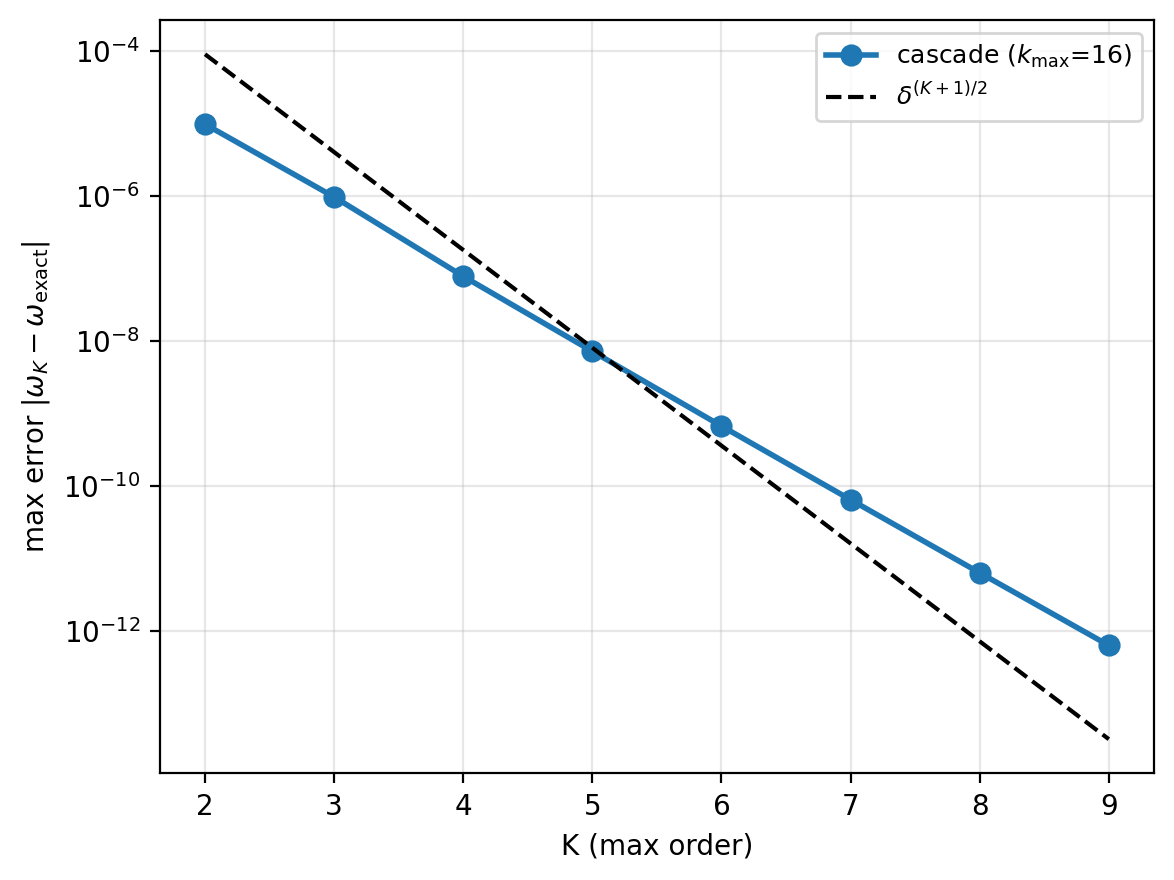}
         \caption{K-convergence of $\omega$ at $\delta = 0.002$}
         \label{fig:dimer_linear1}
     \end{subfigure}
     \hfill
     \begin{subfigure}[b]{0.45\textwidth}
         \centering
         \includegraphics[width=\textwidth]{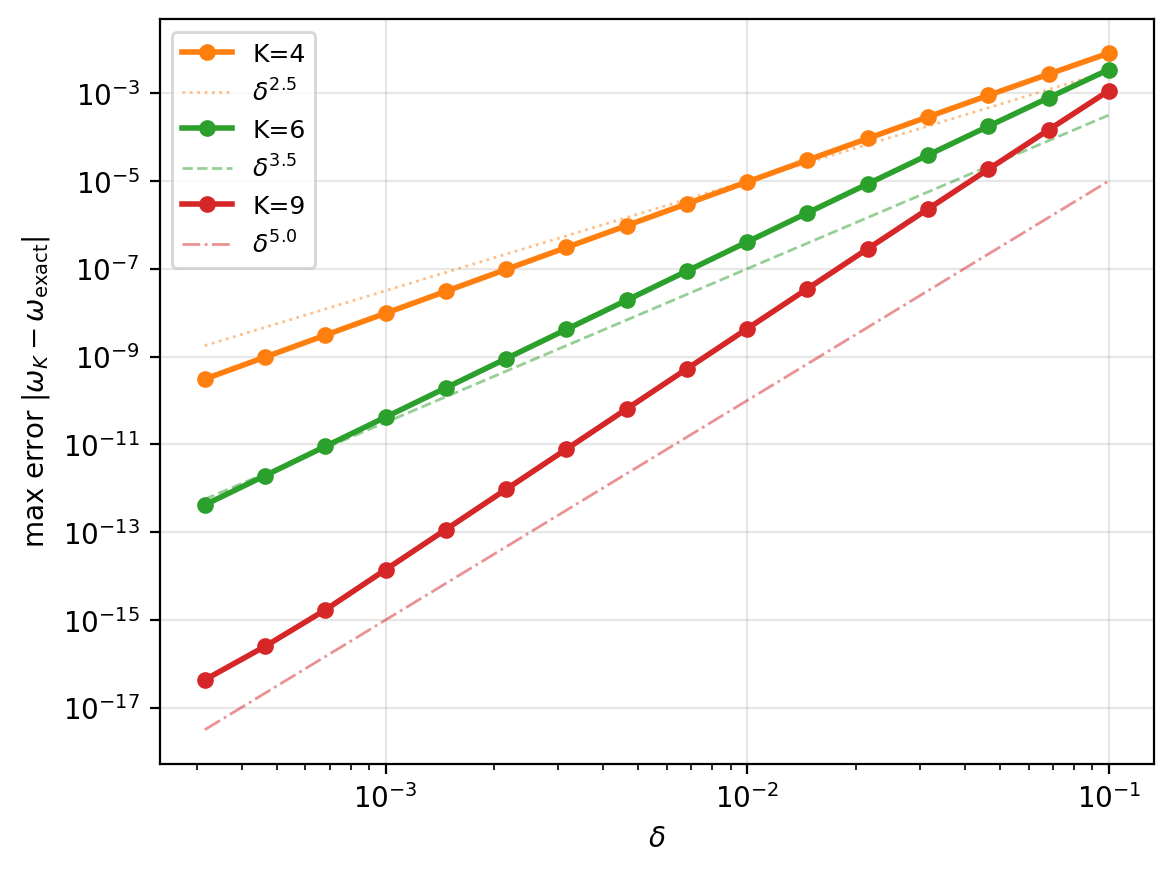}
         \caption{$\delta$-convergence}
         \label{fig:dimer_linear2}
     \end{subfigure}
     \caption{Comparison of perturbative approach versus multipole solution for a two sphere system with radii $r=0.1$, distance $0.3$, interior wave speed $v_1=v_2=1/3.48$ and exterior wave speed $v=1$.}
     \label{fig:dimer_linear}
\end{figure}

We now move on to the nonlinear problem \eqref{eq:nonlinear_problem} for a single resonator. In Figure \ref{fig:single_resonator_nonlinear3}, we see the bifurcation diagram for the discrete approximation. The error plots in Figure \ref{fig:single_resonator_linear1} show excellent agreement with Theorem \ref{thm:approximation_convergence}, while in Figure \ref{fig:single_resonator_nonlinear2} we observe an error plateau around $\delta=0.001$. Further numerical experiments suggest that this is a truncation error.

\begin{figure}[htbp]
     \centering
     \begin{subfigure}[b]{0.45\textwidth}
         \centering
         \includegraphics[width=\textwidth]{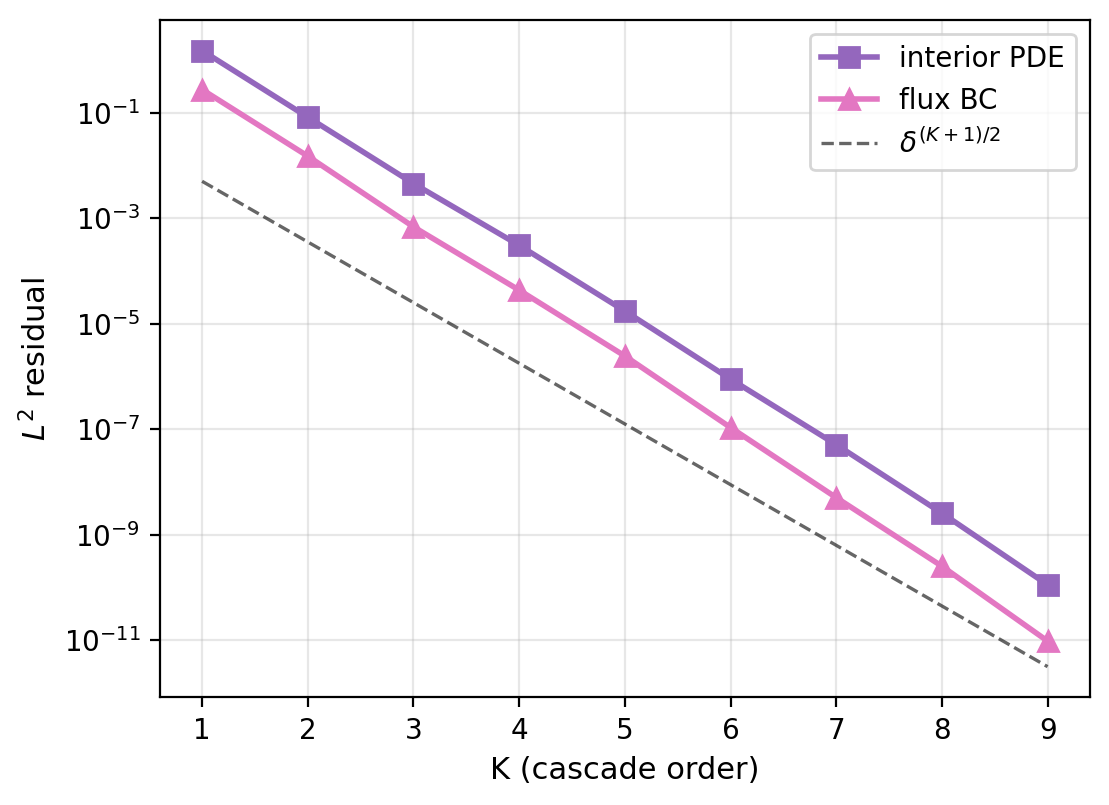}
         \caption{K-convergence of $\omega$ at $\delta = 0.01$}
         \label{fig:single_resonator_nonlinear1}
     \end{subfigure}
     \hfill
     \begin{subfigure}[b]{0.45\textwidth}
         \centering
         \includegraphics[width=\textwidth]{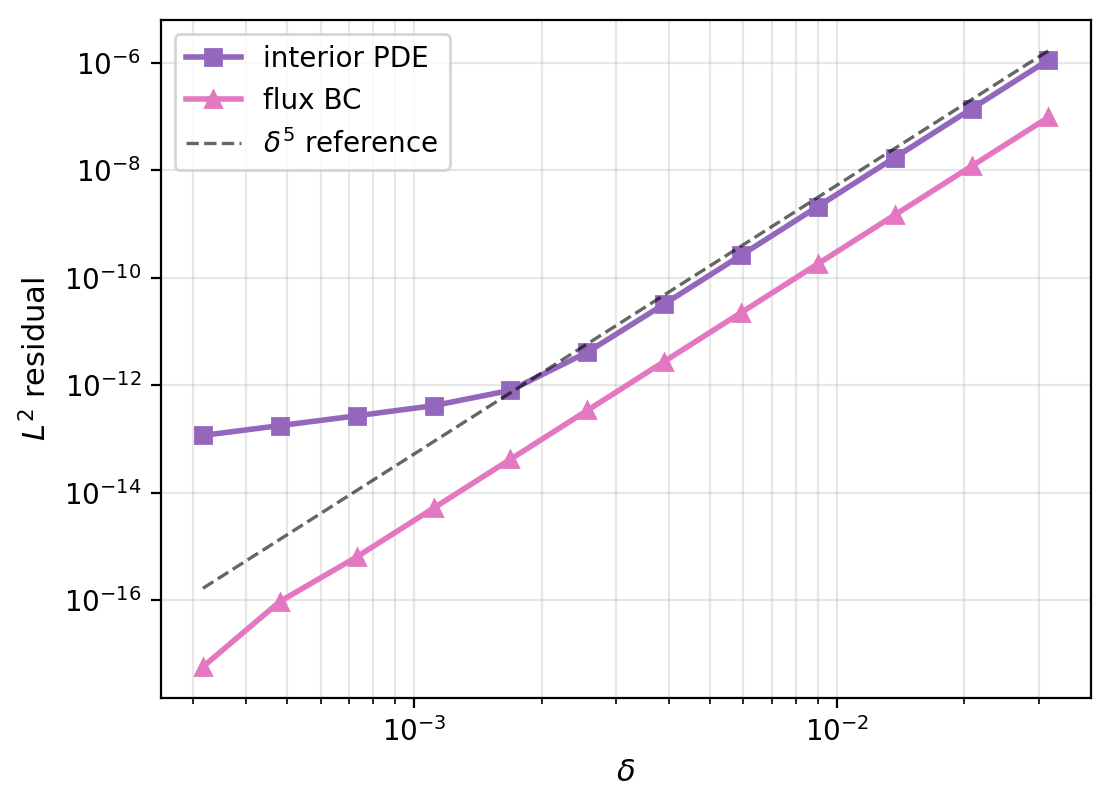}
         \caption{$\delta$-convergence at $K=9$}
         \label{fig:single_resonator_nonlinear2}
     \end{subfigure}

     \vspace{1em} 

     \begin{subfigure}[b]{0.45\textwidth}
         \centering
         \includegraphics[width=\textwidth]{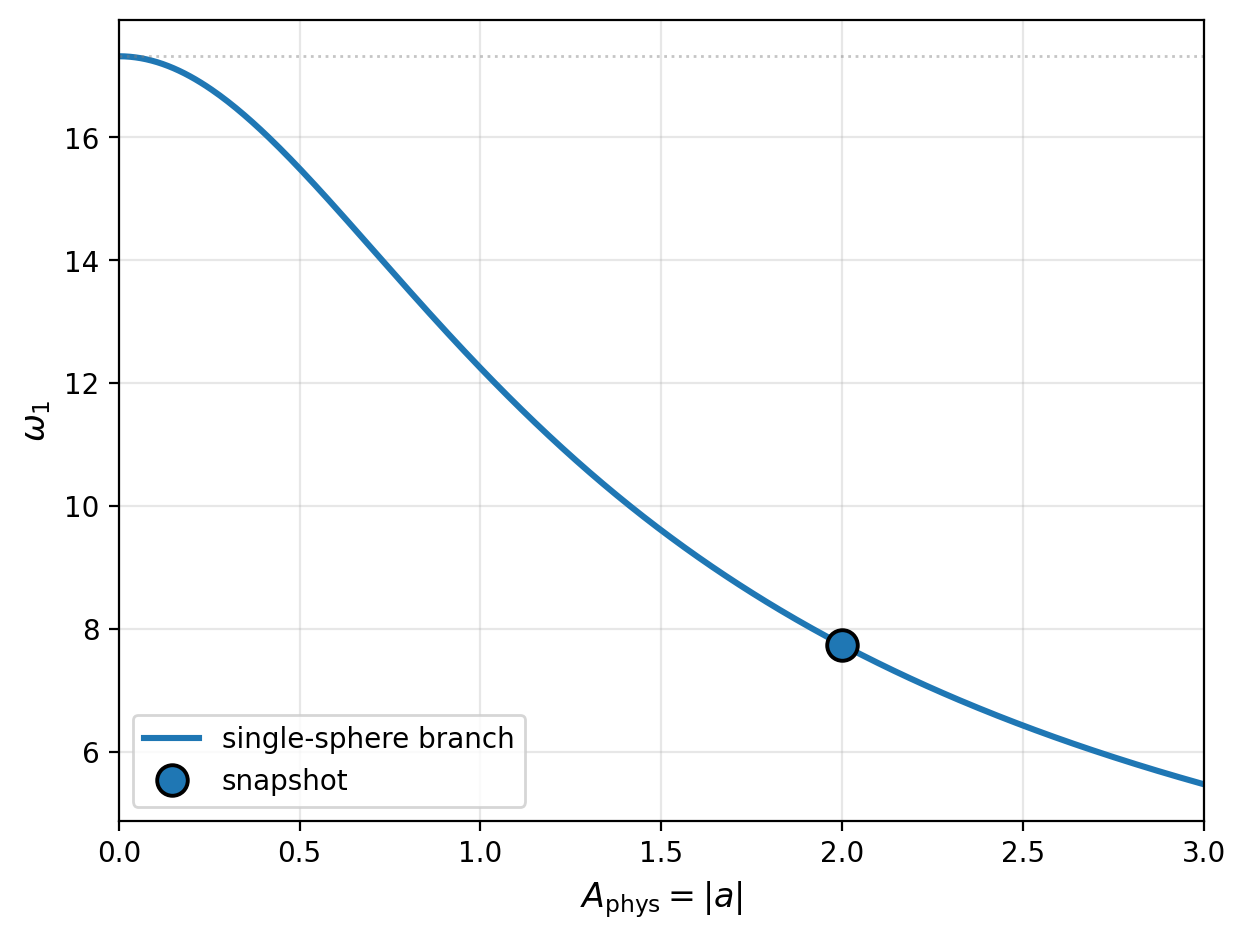}
         \caption{Bifurcation diagram at leading order}
         \label{fig:single_resonator_nonlinear3}
     \end{subfigure}
     
     \caption{Bifurcation diagram and error convergence for a single resonator with radius $r=0.1$, interior wave speed $v_1 = 1/3.48$ and exterior wave speed $v=1$.}
     \label{fig:single_resonator_nonlinear}
\end{figure}

Finally, we turn to the non-subwavelength regime of Section \ref{sec:nonsubwavelength}. In Figure \ref{fig:nonsub_single_resonator_linear}, we study the second resonant frequency of the monomer. The convergence rates of Theorem \ref{thm:approximation_convergence_linear} are clearly visible in plots \ref{fig:nonsub_single_resonator_linear1} and \ref{fig:nonsub_single_resonator_linear2}. The error plateau of \ref{fig:nonsub_single_resonator_linear3} is a truncation error. As before, further numerical experiments have shown that the plateau decreases when increasing $k_{\max}$ and $l_{\max}$. Another example for the linear, non-subwavelength regime is shown in Figure \ref{fig:nonsub_dimer_linear}, where the second resonance of a dimer system is studied. The last Figure of this Section, \ref{fig:nonsub_nonlinear_single} showcases the error convergence with respect to $\delta$, that is, the residuals are divided by $\delta$. 

\begin{figure}[htbp]
     \centering
     \begin{subfigure}[b]{0.45\textwidth}
         \centering
         \includegraphics[width=\textwidth]{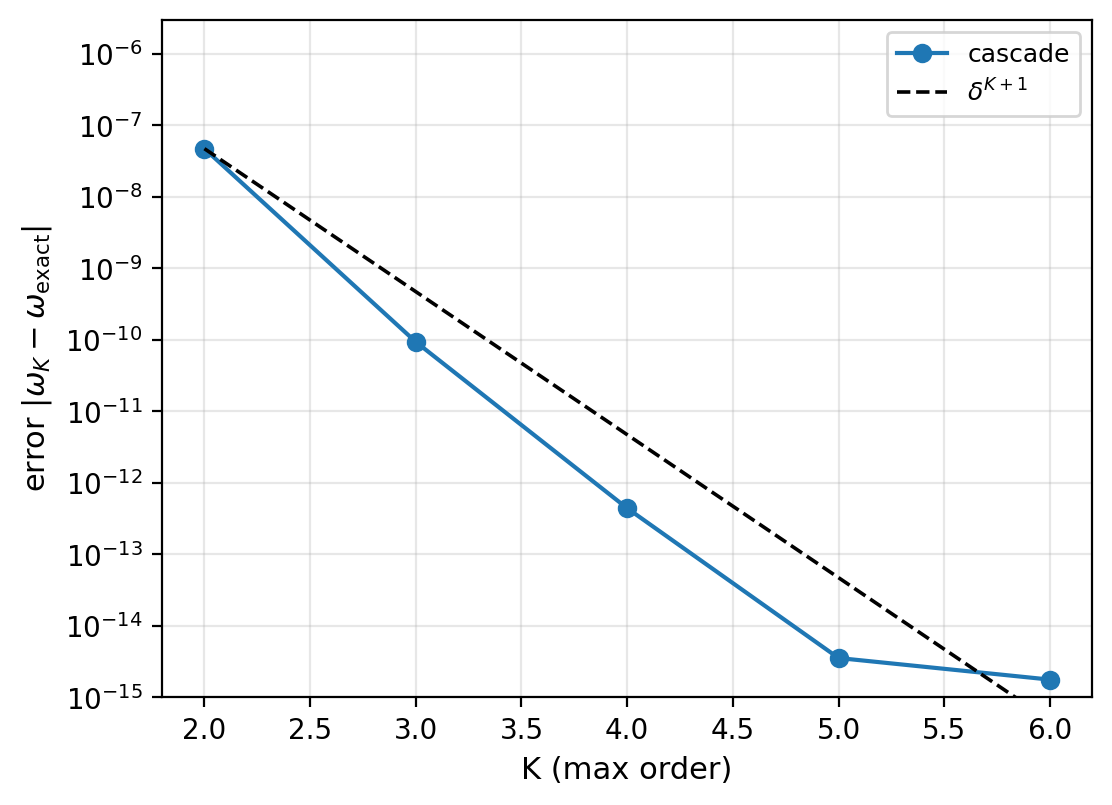}
         \caption{K-convergence of $\omega$ at $\delta = 0.01$}
         \label{fig:nonsub_single_resonator_linear1}
     \end{subfigure}
     \hfill
     \begin{subfigure}[b]{0.45\textwidth}
         \centering
         \includegraphics[width=\textwidth]{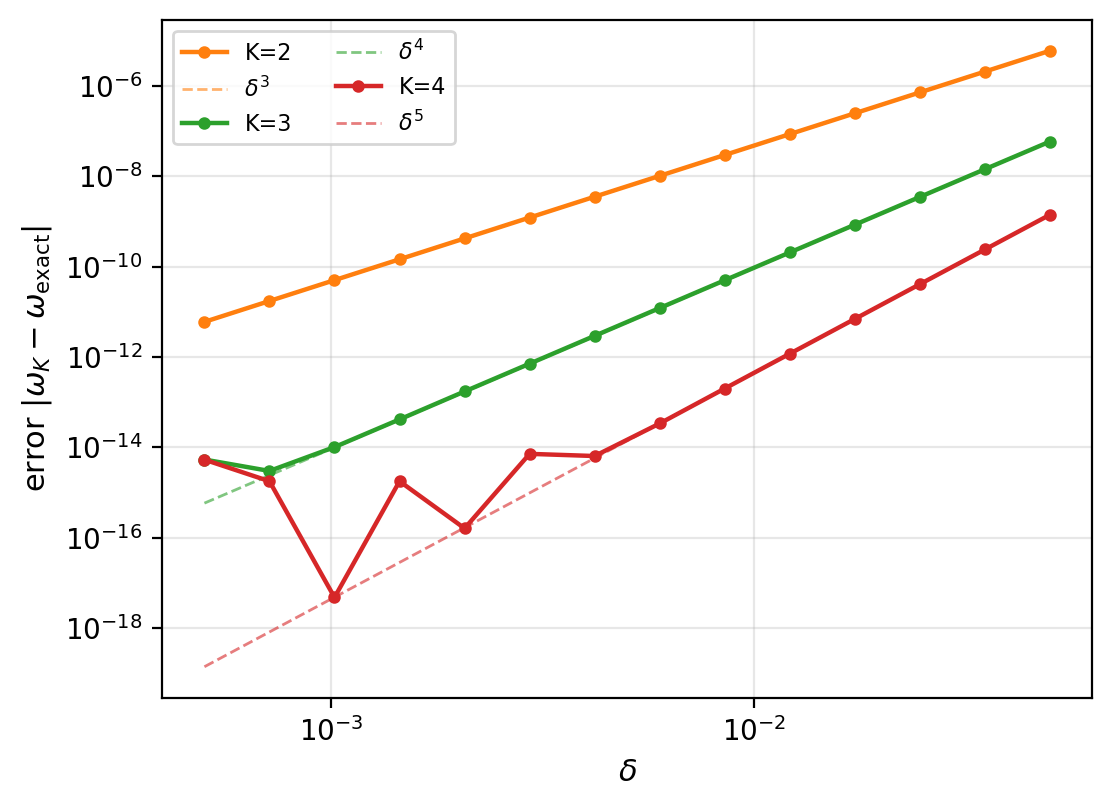}
         \caption{$\delta$-convergence}
         \label{fig:nonsub_single_resonator_linear2}
     \end{subfigure}

     \vspace{1em} 

     \begin{subfigure}[b]{0.45\textwidth}
         \centering
         \includegraphics[width=\textwidth]{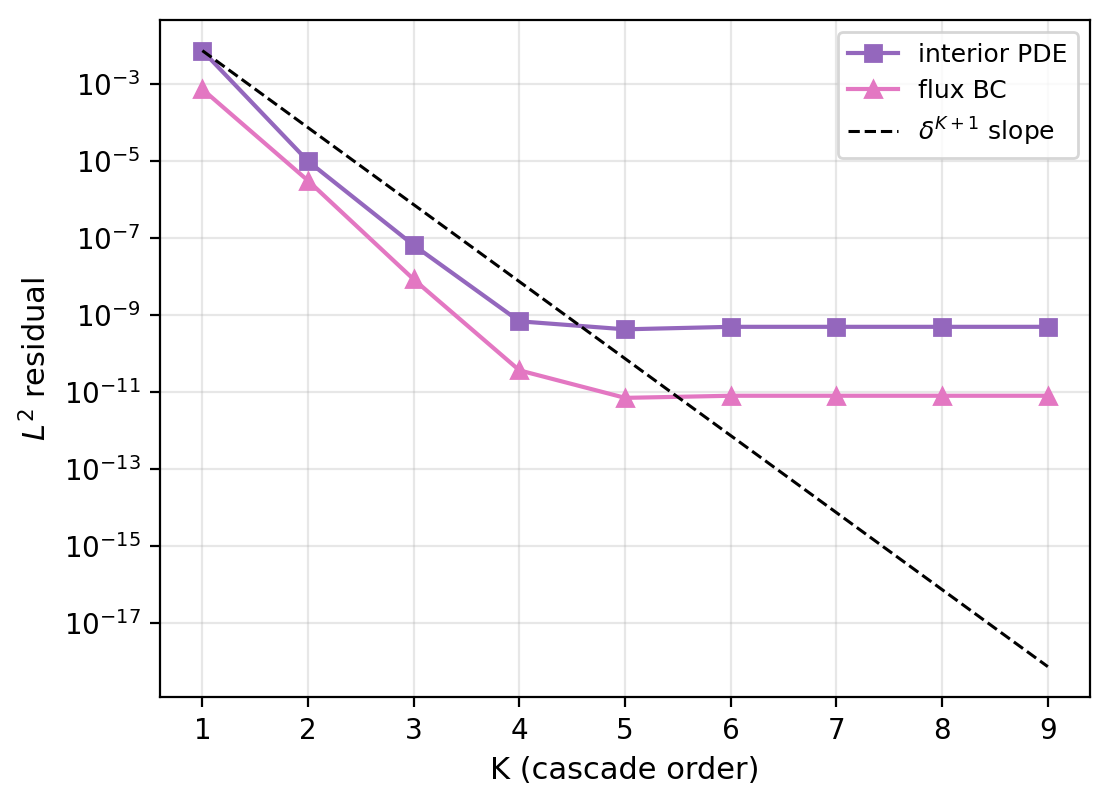}
         \caption{Operator residuals at $\delta = 0.01$}
         \label{fig:nonsub_single_resonator_linear3}
     \end{subfigure}
     
     \caption{Comparison of perturbative approach versus exact solution of the second resonance frequency for a single resonator with radius $r=0.1$, interior wave speed $v_1 = 1/3.48$ and exterior wave speed $v=1$.}
     \label{fig:nonsub_single_resonator_linear}
\end{figure}
\begin{figure}[htbp]
     \centering
     \begin{subfigure}[b]{0.45\textwidth}
         \centering
         \includegraphics[width=\textwidth]{figures/nonsub_monomer_linear_Kconvergence_a0p1_v0p2874_l0p1_sigma0_delta0p01.png}
         \caption{K-convergence of $\omega$ at $\delta = 0.01$}
         \label{fig:nonsub_dimer_linear1}
     \end{subfigure}
     \hfill
     \begin{subfigure}[b]{0.45\textwidth}
         \centering
         \includegraphics[width=\textwidth]{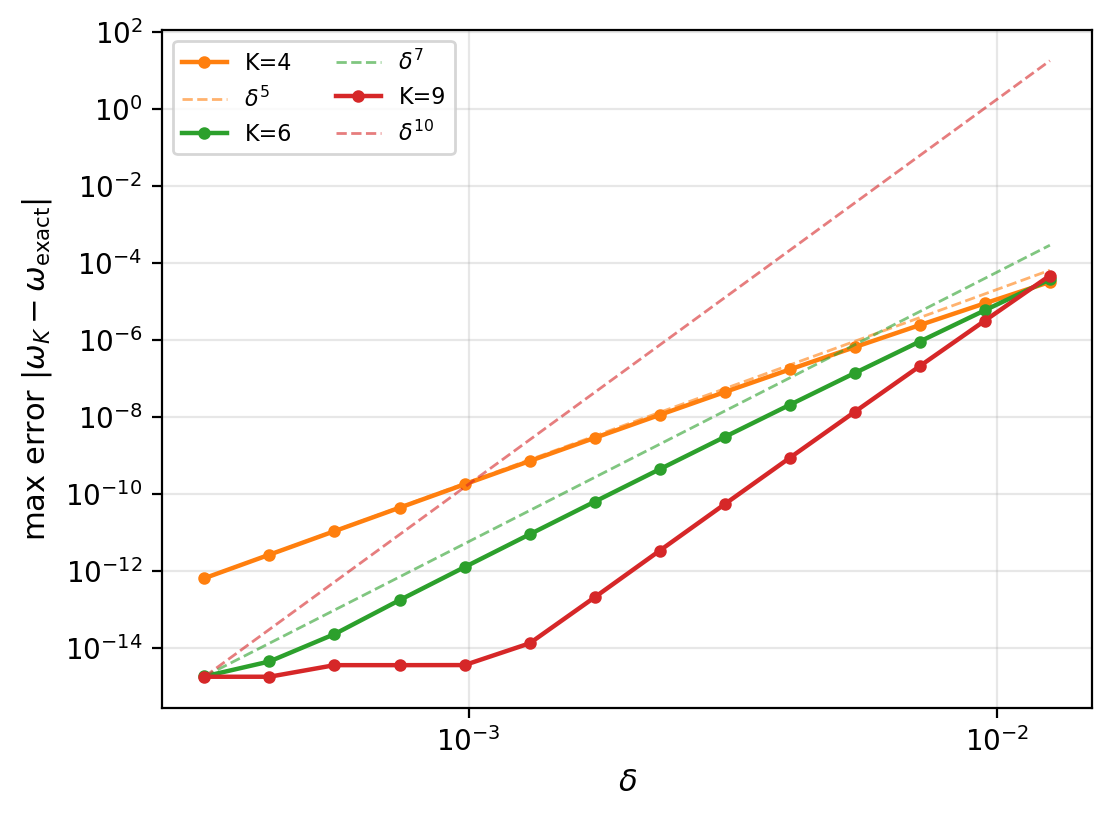}
         \caption{$\delta$-convergence}
         \label{fig:nonsub_dimer_linear2}
     \end{subfigure}
     \caption{Comparison of perturbative approach versus multipole solution of the second resonance for a two sphere system with radii $r=0.1$, distance $0.3$, interior wave speed $v_1=v_2=1/3.48$ and exterior wave speed $v=1$.}
     \label{fig:nonsub_dimer_linear}
\end{figure}
\begin{figure}[htbp]
     \centering
     \begin{subfigure}[b]{0.45\textwidth}
         \centering
         \includegraphics[width=\textwidth]{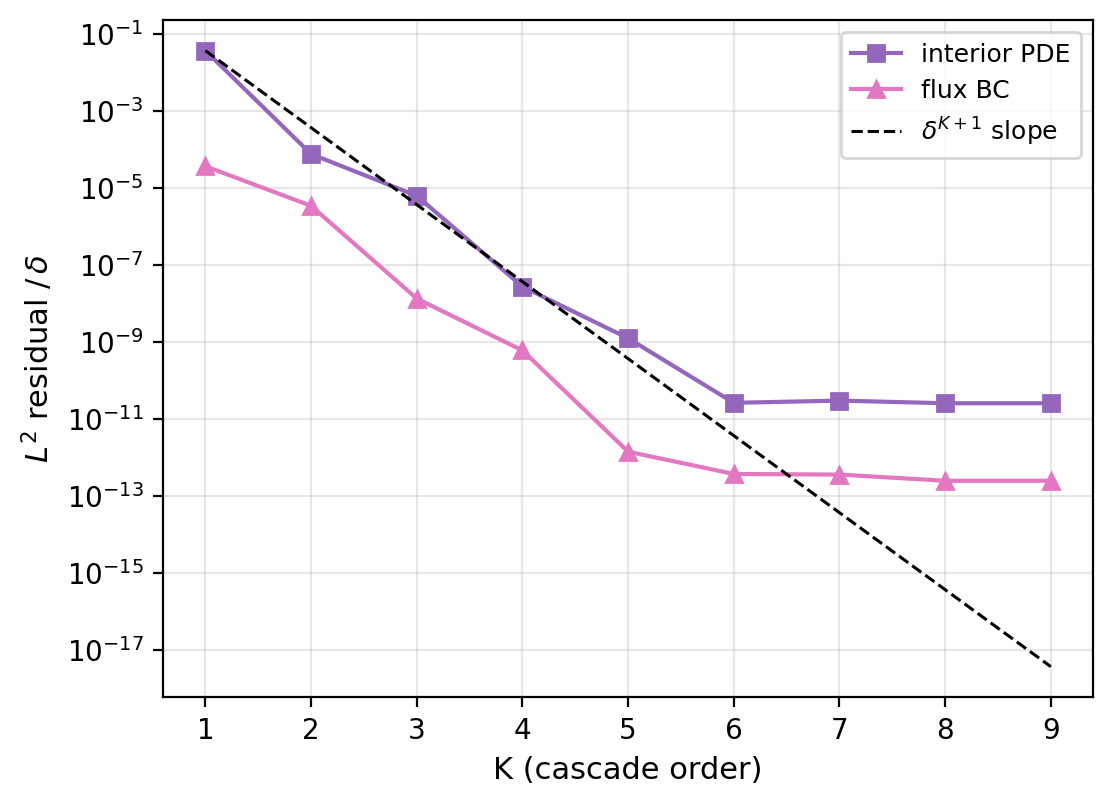}
         \caption{K-convergence of $\omega$ at $\delta = 0.01$}
         \label{fig:nonsub_nonlinear_single1}
     \end{subfigure}
     \hfill
     \begin{subfigure}[b]{0.45\textwidth}
         \centering
         \includegraphics[width=\textwidth]{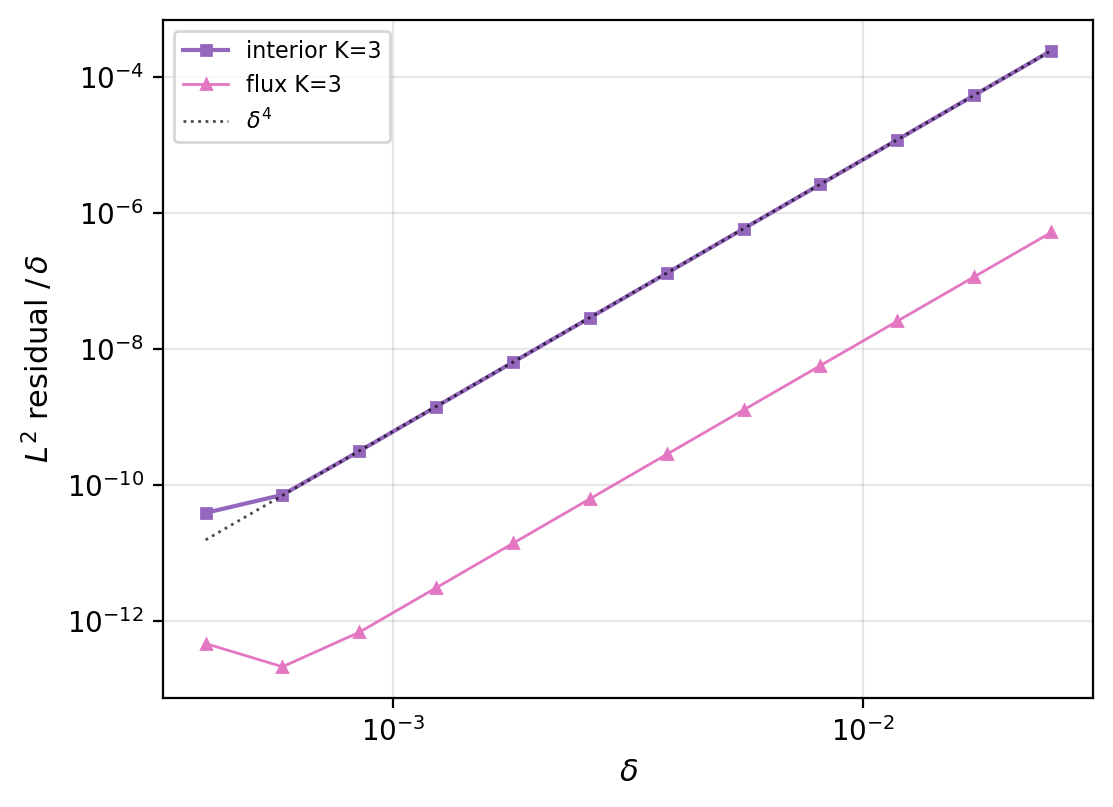}x
         \caption{$\delta$-convergence}
         \label{fig:nonsub_nonlinear_single2}
     \end{subfigure}

     \vspace{1em} 

     \begin{subfigure}[b]{0.45\textwidth}
         \centering
         \includegraphics[width=\textwidth]{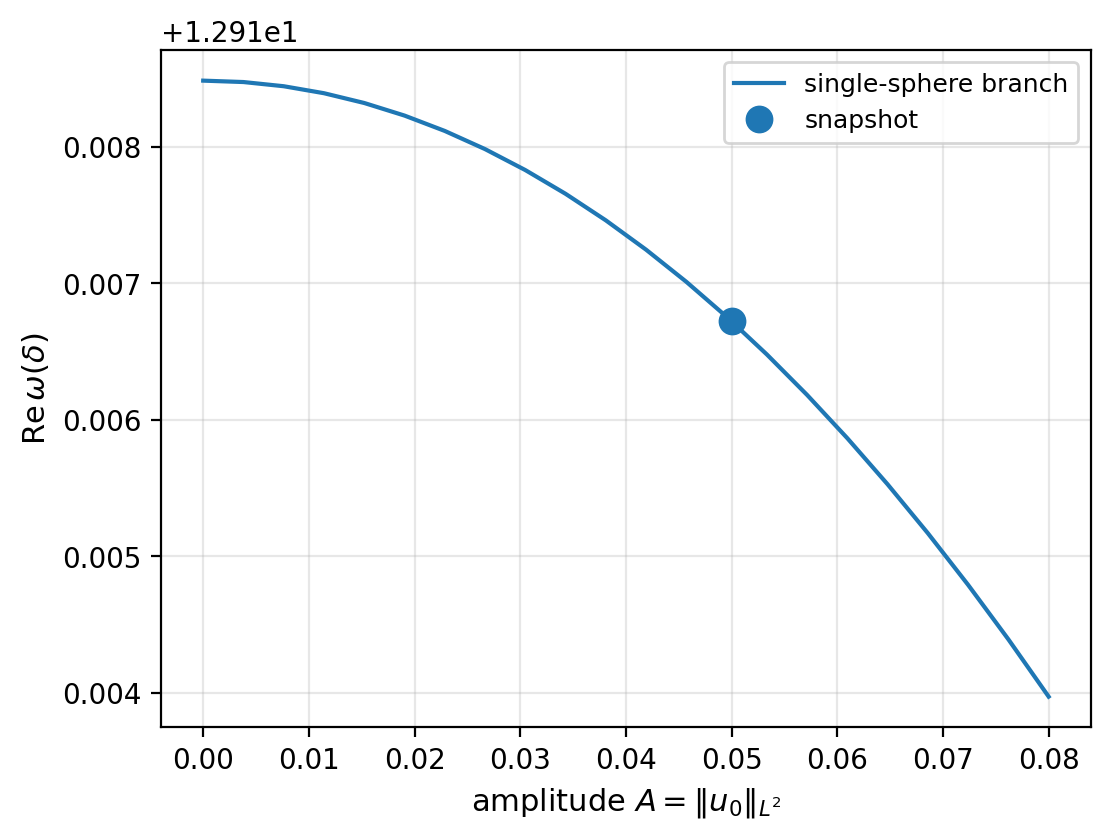}
         \caption{Bifurcation diagram of the discrete system}
         \label{fig:nonsub_nonlinear_single3}
     \end{subfigure}
     
     \caption{Small amplitude soliton bifurcating from the second resonance frequency for a single resonator with radius $r=0.1$, interior wave speed $v_1 = 1/3.48$ and exterior wave speed $v=1$.}
     \label{fig:nonsub_nonlinear_single}
\end{figure}

\subsection{Symmetry Breaking Bifurcations in Symmetric Dimers}
In this subsection, we study the nonlinear problem \eqref{eq:nonlinear_problem} for a symmetric dimer setup of two identical spherical resonators of radius $r=0.1$, separated by a distance of $d=0.2$, with interior wave speed $v_1=v_2 = 1$. By the reflection symmetry of the dimer, 
the capacitance matrix has the form
\begin{equation}\label{eq:dimer_capacitance}
    \mathcal{C} = \begin{pmatrix}a & b \\ b & a \end{pmatrix},
\end{equation}
with eigenvalues $\lambda_\pm = v_0^2(a \pm b)/\lvert D_0\rvert$  corresponding to the symmetric and antisymmetric 
eigenmodes. Moreover, the entries $a$ and 
$b$ admit closed-form series expressions in bispherical 
coordinates~\cite{lekner2012electrostatics}, valid for all separations 
including nearly touching configurations. This explicit form enables an 
analytical study of the nonlinear branches. 
Substituting $\mathbf{a} = (x, y)$ into the discrete nonlinear 
capacitance system~\eqref{eq:discrete_nonlinear} and writing $\lambda = 
\omega^2$ for brevity, we obtain
\begin{equation}\label{eq:dimer_real_system}
    a x + b y = \lambda(x + \sigma x^3), \qquad 
    b x + a y = \lambda(y + \sigma y^3),
\end{equation}
where, by $S^1$ invariance, we may take $x, y \in \mathbb{R}$. Adding 
and subtracting~\eqref{eq:dimer_real_system} yields the factored system
\begin{align}
    (x+y)\bigl[(a+b) - \lambda - \lambda\sigma(x^2 - xy + y^2)\bigr] &= 0, \\
    (x-y)\bigl[(a-b) - \lambda - \lambda\sigma(x^2 + xy + y^2)\bigr] &= 0,
\end{align}
from which three nontrivial branches emerge:

\textbf{Symmetric branch} ($x = y$). The first factor vanishes, yielding
\[
    \sigma x^2 = \frac{a + b - \lambda}{\lambda}.
\]
This branch bifurcates from the trivial solution at $\lambda = a + b$ 
along the symmetric eigendirection.

\textbf{Antisymmetric branch} ($x = -y$). The second factor vanishes, 
yielding
\[
    \sigma x^2 = \frac{a - b - \lambda}{\lambda}.
\]
This branch bifurcates from the trivial solution at $\lambda = a - b$ 
along the antisymmetric eigendirection.

\textbf{Asymmetric branch} ($x \neq \pm y$). Both factors vanish 
simultaneously, yielding
\[
    xy = -\frac{b}{\sigma \lambda}, \qquad 
    x^2 + y^2 = \frac{a - \lambda}{\sigma \lambda}.
\]
This branch is real and distinct from the previous two if and only if 
$(a - \lambda)^2 > 4 b^2$ and the resulting quadratic for $x^2, y^2$ 
admits two positive roots. It corresponds to states localised 
predominantly on one resonator and emerges as a secondary bifurcation 
from one of the primary branches at a critical amplitude.

The emergence of symmetry breaking bifurcations is well studied \cite{eilbeck1985discrete, kirr2008symmetry, aschbacher2002symmetry}. Indeed, the above characterisation has been done in \cite{eilbeck1985discrete} for a similar setup, moreover, therein the authors also explicitly treat trimer and quadrimer like systems. Symmetry breaking bifurcations were also proven to emerge in a high-refractive framework \cite{ammari2026dielectric}.

In Figure \ref{fig:dimer_bifurcation}, we can see the resulting bifurcation diagram, with the symmetry breaking mode bifurcating from the symmetric mode. Figures \ref{fig:dimer_antisym}, \ref{fig:dimer_sym} and \ref{fig:dimer_asym} show error convergence as well as field plots for the three points marked in Figure \ref{fig:dimer_bifurcation}.
\begin{figure}
    \centering
    \includegraphics[width=0.5\linewidth]{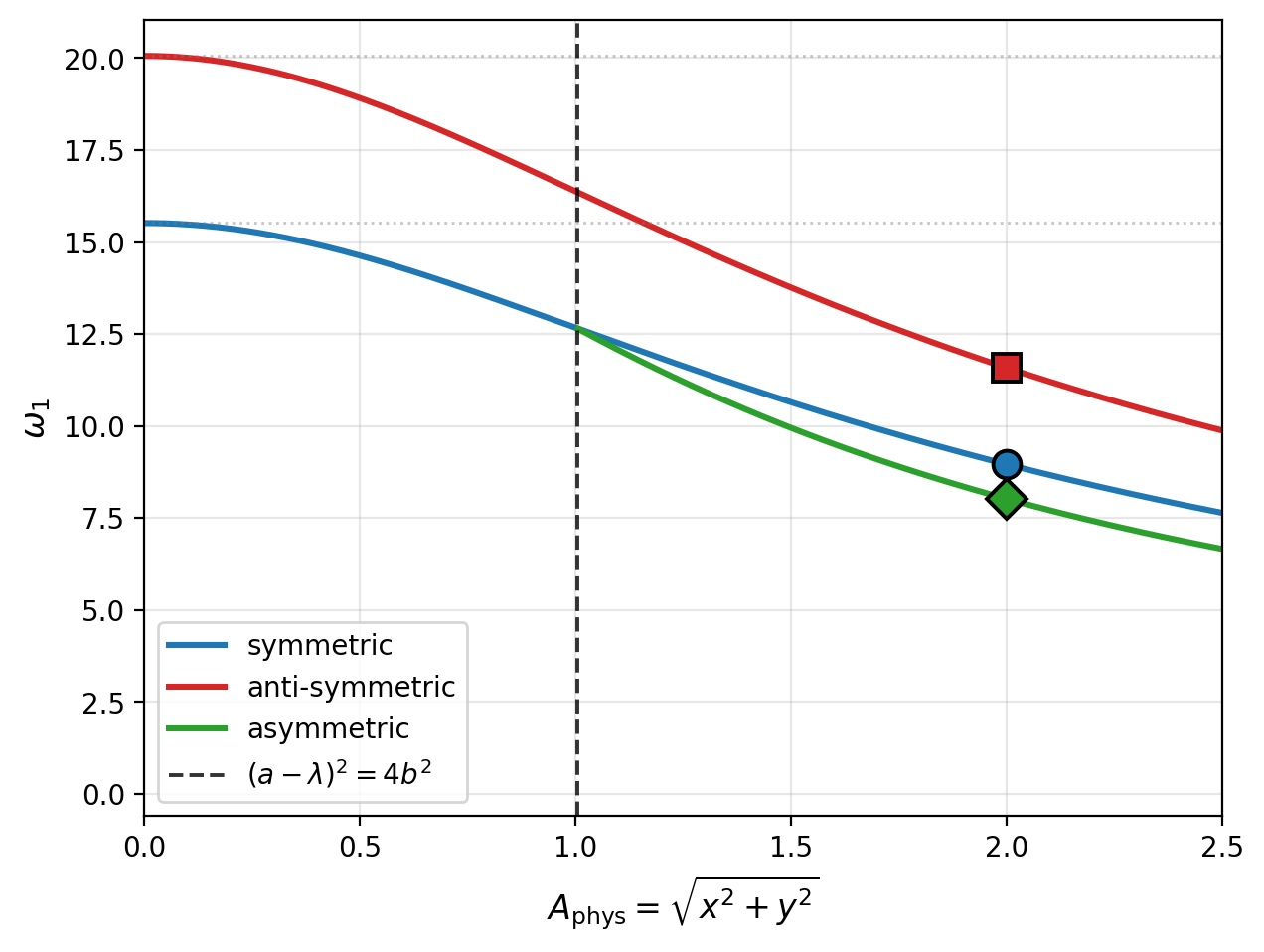}
    \caption{Bifurcation diagram for discrete dimer model, the markers denote the snapshots for the subsequent plots.}
    \label{fig:dimer_bifurcation}
\end{figure}

\begin{figure}[htbp]
     \centering
     \begin{subfigure}[b]{0.45\textwidth}
         \centering
         \includegraphics[width=\textwidth]{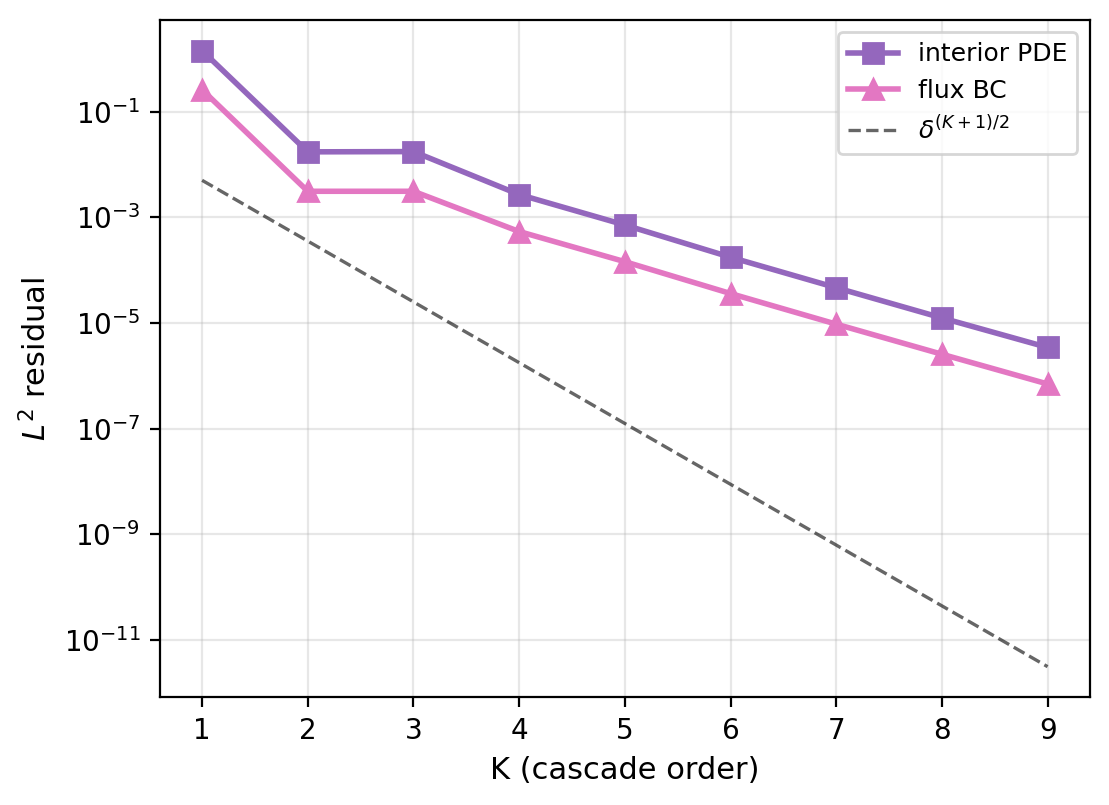}
         \caption{K-convergence of $\omega$ at $\delta = 0.005$}
         \label{fig:dimer_antisym1}
     \end{subfigure}
     \hfill
     \begin{subfigure}[b]{0.45\textwidth}
         \centering
         \includegraphics[width=\textwidth]{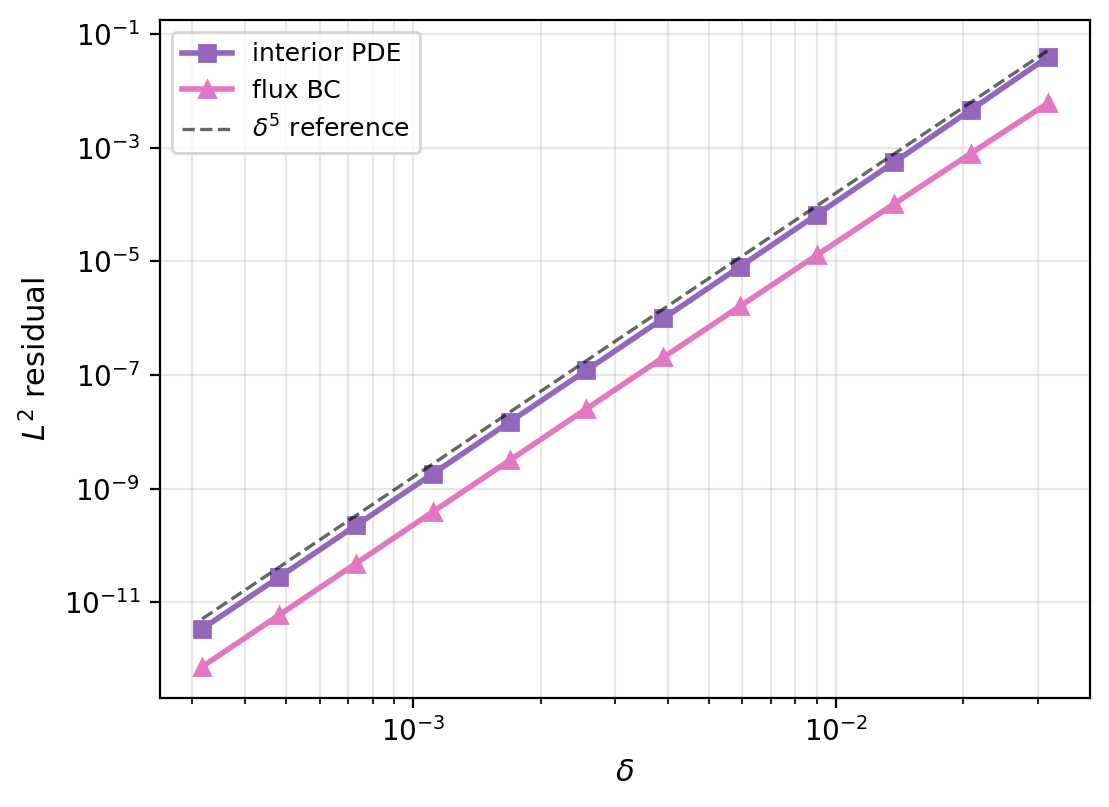}
         \caption{$\delta$-convergence at $K=9$}
         \label{fig:dimer_antisym2}
     \end{subfigure}

     \vspace{1em} 

     \begin{subfigure}[b]{0.45\textwidth}
         \centering
         \includegraphics[width=\textwidth]{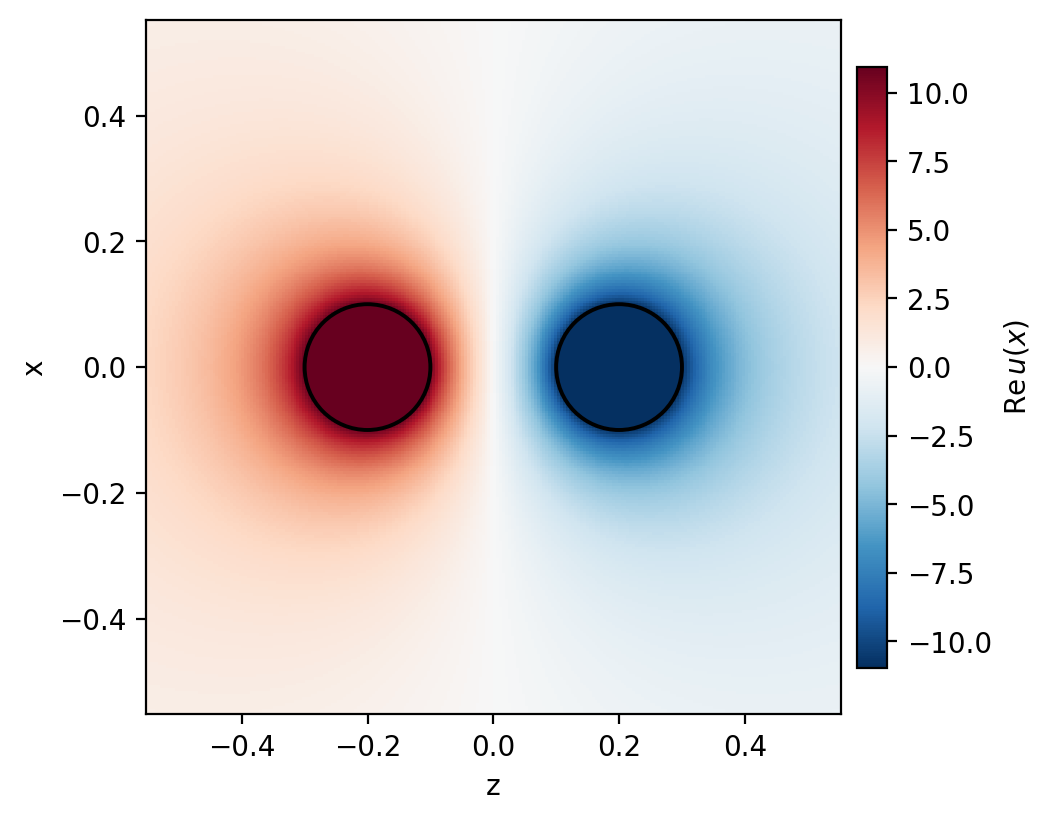}
         \caption{Plot of $\Re(u)$ at $\delta=0.005$}
         \label{fig:dimer_antisym3}
     \end{subfigure}
     
     \caption{Error and field plots for the antisymmetric mode. Note the that in figure \ref{fig:dimer_antisym1} the rate of convergence does not match the reference slope. Numerical evidence suggests that this is due to the exponential growth of the coefficient $\omega_n,u_n$.}
     \label{fig:dimer_antisym}
\end{figure}

\begin{figure}[htbp]
     \centering
     \begin{subfigure}[b]{0.45\textwidth}
         \centering
         \includegraphics[width=\textwidth]{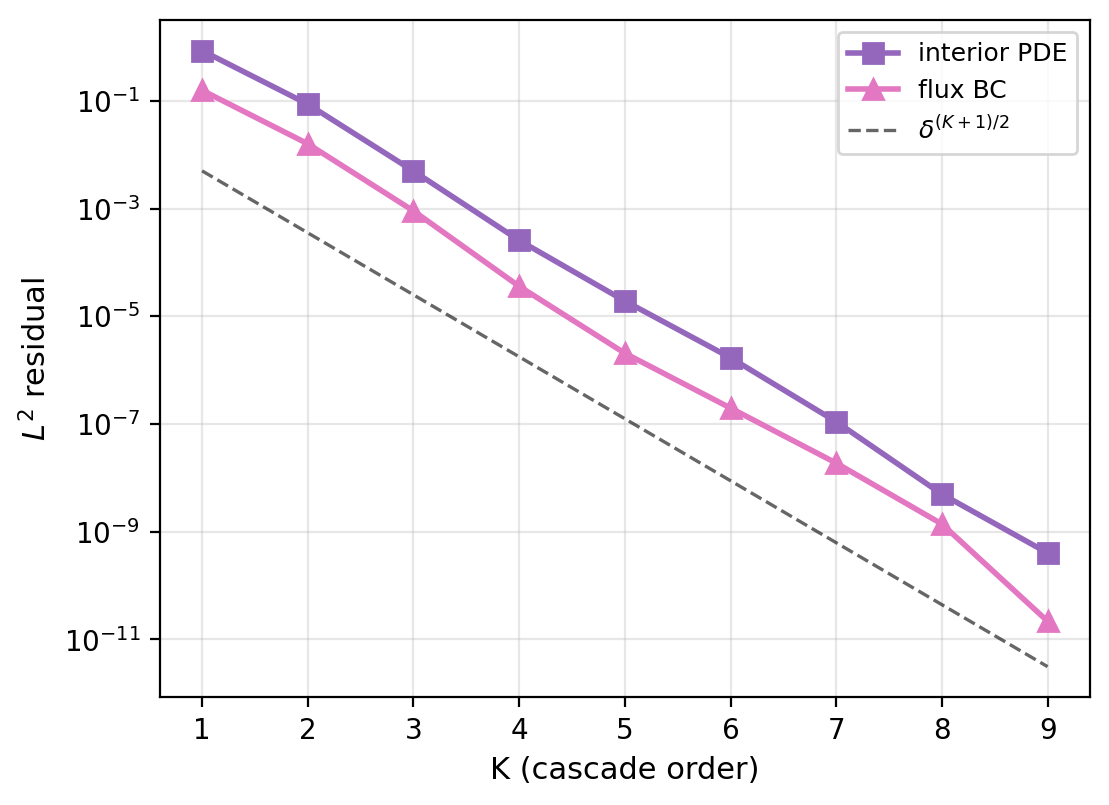}
         \caption{K-convergence of $\omega$ at $\delta = 0.01$}
         \label{fig:dimer_sym1}
     \end{subfigure}
     \hfill
     \begin{subfigure}[b]{0.45\textwidth}
         \centering
         \includegraphics[width=\textwidth]{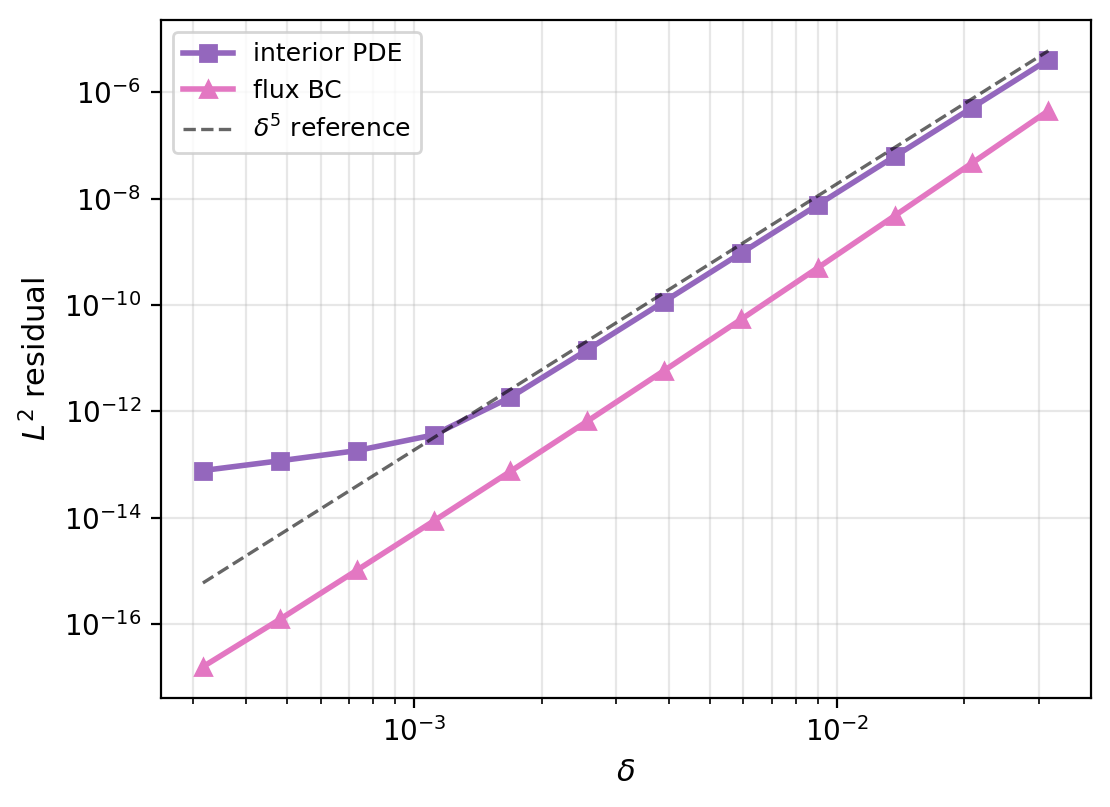}
         \caption{$\delta$-convergence at $K=9$}
         \label{fig:dimer_sym2}
     \end{subfigure}

     \vspace{1em} 

     \begin{subfigure}[b]{0.45\textwidth}
         \centering
         \includegraphics[width=\textwidth]{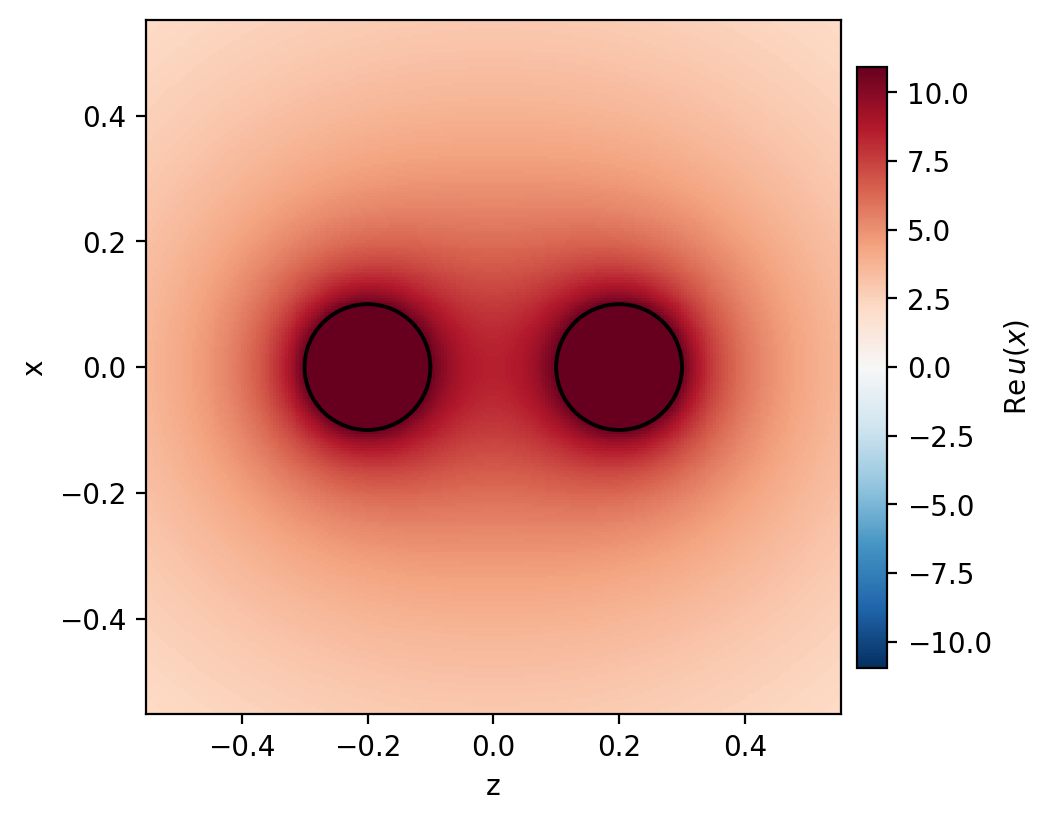}
         \caption{Plot of $\Re(u)$ at $\delta = 0.01$}
         \label{fig:dimer_sym3}
     \end{subfigure}
     
     \caption{Error and field plots for the symmetric mode. As in Figure \ref{fig:single_resonator_nonlinear2}, the error plateau in \ref{fig:dimer_sym2} is most likely due to truncation error.}
     \label{fig:dimer_sym}
\end{figure}
\begin{figure}[htbp]
     \centering
     \begin{subfigure}[b]{0.45\textwidth}
         \centering
         \includegraphics[width=\textwidth]{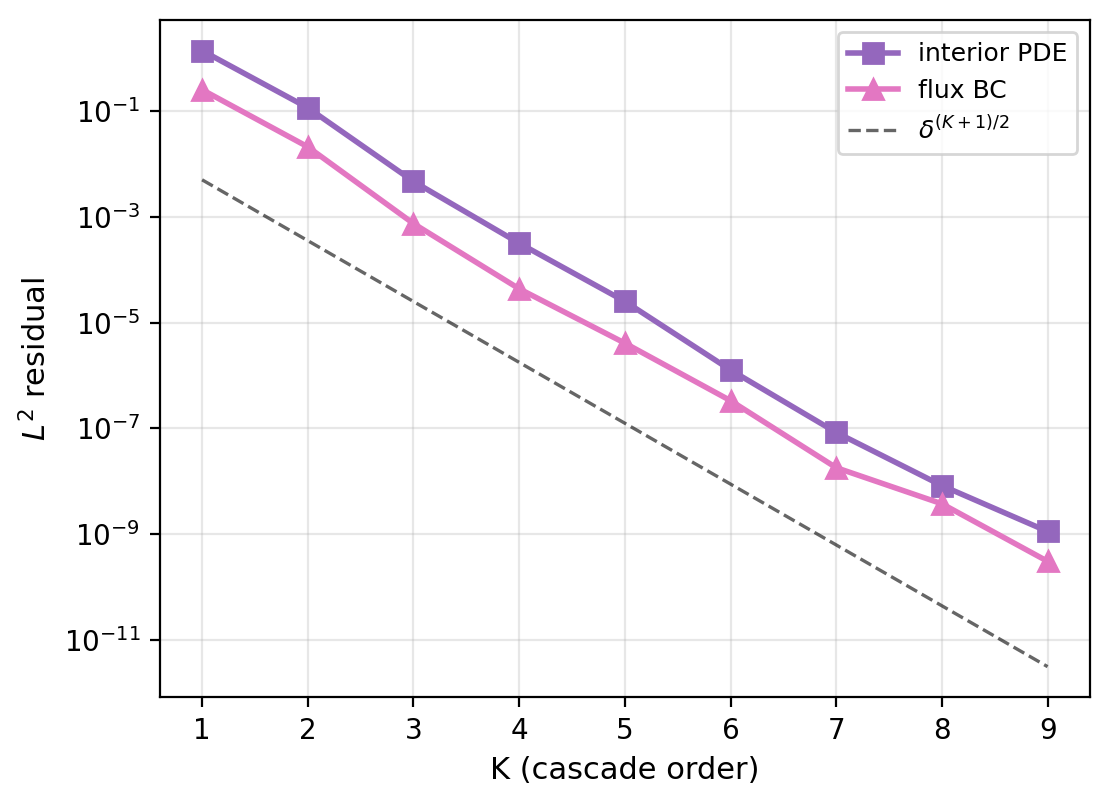}
         \caption{K-convergence of $\omega$ at $\delta = 0.01$}
         \label{fig:dimer_asym1}
     \end{subfigure}
     \hfill
     \begin{subfigure}[b]{0.45\textwidth}
         \centering
         \includegraphics[width=\textwidth]{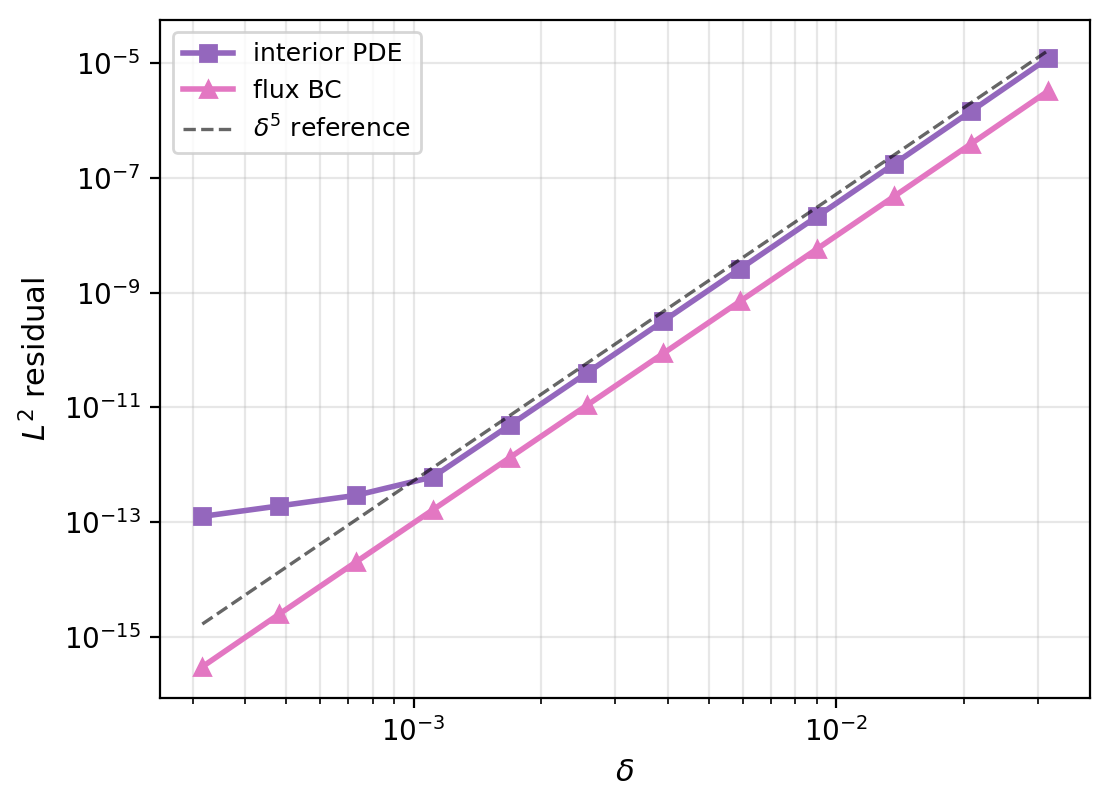}
         \caption{$\delta$-convergence at $K=9$}
         \label{fig:dimer_asym2}
     \end{subfigure}

     \vspace{1em} 

     \begin{subfigure}[b]{0.45\textwidth}
         \centering
         \includegraphics[width=\textwidth]{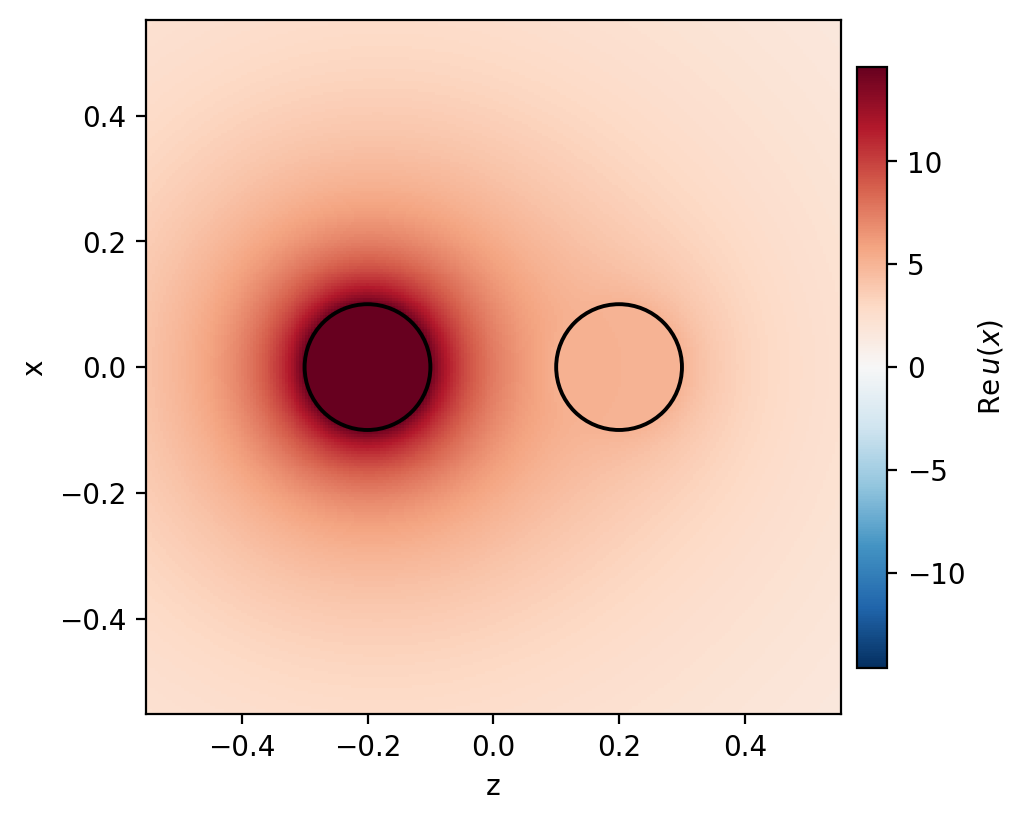}
         \caption{Plot of $\Re(u)$ at $\delta = 0.01$}
         \label{fig:dimer_asym3}
     \end{subfigure}
     
     \caption{Error and field plots for the asymmetric mode. One again observes the familiar error plateau in Figure  \ref{fig:dimer_asym2}.}
     \label{fig:dimer_asym}
\end{figure}

\section{Conclusion}

In this paper, we have studied the nonlinear Helmholtz problem with 
outgoing radiation condition and cubic nonlinearity in the high-contrast 
regime. Using perturbative techniques, we have established a two-way 
correspondence between the emerging solitons and a discrete model: under 
a non-degeneracy assumption on the discrete data, every solution of the 
nonlinear capacitance system lifts to a continuous soliton, and every 
continuous soliton with the natural subwavelength scaling reduces to a 
discrete one. The capacitance matrix therefore faithfully captures the 
nonlinear resonant behaviour of the full system in this regime.

The result is constructive: the cascade equations yield an iterative 
algorithm for computing correction terms to arbitrary order, in both 
the linear and nonlinear, subwavelength and non-subwavelength regimes. Moreover, analyticity is shown in $\sqrt{\delta}$, respectively in $\delta$. To 
our knowledge, this is the first existence result for both subwavelength 
and non-subwavelength solitons in the high-contrast regime. Beyond 
establishing the correspondence itself, the reduction to a finite 
discrete system makes techniques from the discrete self-trapping and 
discrete nonlinear Schr\"odinger literature applicable to subwavelength 
resonance problems. As a first illustration, we have shown that the 
symmetric dimer admits a symmetry-breaking bifurcation with a criterion 
expressed cleanly in terms of the eigenvalues of the generalised 
capacitance matrix.

\section*{Acknowledgements}
C. Thalhammer was supported in part by the Swiss National Science Foundation grant number 200021-236472. 

\section*{Data Availability}
The code used to generate the Figures is openly available at \url{https://github.com/cthalhammer/perturbative-nonlinear-helmholtz}.

\appendix

\printbibliography

\end{document}